\documentclass[12pt,leqno]{article}
\usepackage{amsfonts}
\usepackage{amsmath, amsthm, amsfonts, amssymb, color}
\usepackage{mathrsfs,enumitem}
\usepackage{color}
\newtheorem{thm}{Theorem}[section]
\newtheorem{cor}[thm]{Corollary}
\newtheorem{lem}[thm]{Lemma}
\newtheorem{prp}[thm]{Proposition}

\newtheorem{rem}[thm]{Remark}
\newtheorem{defn}[thm]{Definition}
\theoremstyle{definition}
\newcommand{\scr}[1]{\mathscr #1}
\definecolor{wco}{rgb}{0.5,0.2,0.3}

\numberwithin{equation}{section} \theoremstyle{remark}

\newcommand{\ua}{\uparrow}

\title{{\bf Linearization of Nonlinear Fokker-Planck Equations and Applications}\footnote{Supported in
 part by  NNSFC (11771326, 11831014, 11921001) and the DFG through the CRC 1283.} }
\author{
{\bf    Panpan Ren$^{c), d)}$, Michael R\"ockner$^{b,e)}$, Feng-Yu Wang$^{a),c)}$  }\\
\footnotesize{$^{a)}$ Center for Applied Mathematics, Tianjin University, Tianjin 300072, China }\\
\footnotesize{$^{b)}$ Department of Mathematics, Bielefeld
University, D-33501 Bielefeld, Germany}\\
 \footnotesize{ $^{c)}$ Department of Mathematics,
Swansea University, Bay Campus,
SA1 8EN, United Kingdom}\\
\footnotesize{$^{d)}$ Mathematical Institute, University of Oxford, Woodstock Road, OX2 6GG, United Kingdom}\\
\footnotesize{ $^{e)}$ Academy of Mathematics and System Sciences, Chinese Academy of Science, Beijing, China}\\
 \footnotesize{  Panpan.ren@maths.ox.ac.uk;
roeckner@math.uni-bielefeld.de; wangfy@tju.edu.cn}}
\begin{document}
\allowdisplaybreaks
\def\R{\mathbb R}  \def\ff{\frac} \def\ss{\sqrt} \def\B{\mathbf
B}
\def\N{\mathbb N} \def\kk{\kappa} \def\m{{\bf m}}
\def\ee{\varepsilon}\def\ddd{D^*}
\def\dd{\delta} \def\DD{\Delta} \def\vv{\varepsilon} \def\rr{\rho}
\def\<{\langle} \def\>{\rangle} \def\GG{\Gamma} \def\gg{\gamma}
  \def\nn{\nabla} \def\pp{\partial} \def\E{\mathbb E}
\def\d{\text{\rm{d}}} \def\bb{\beta} \def\aa{\alpha} \def\D{\scr D}
  \def\si{\sigma} \def\ess{\text{\rm{ess}}}
\def\beg{\begin} \def\beq{\begin{equation}}  \def\F{\scr F}
\def\Ric{\scr Ric} \def\Hess{\text{\rm{Hess}}}
\def\e{\text{\rm{e}}} \def\ua{\underline a} \def\OO{\Omega}  \def\oo{\omega}
 \def\tt{\tilde}
\def\cut{\text{\rm{cut}}} \def\P{\mathbb P} \def\ifn{I_n(f^{\bigotimes n})}
\def\C{\scr C}      \def\aaa{\mathbf{r}}     \def\r{r}
\def\gap{\text{\rm{gap}}} \def\prr{\pi_{{\bf m},\varrho}}  \def\r{\mathbf r}
\def\Z{\mathbb Z} \def\vrr{\varrho} \def\ll{\lambda}
\def\L{\scr L}\def\Tt{\tt} \def\TT{\tt}\def\II{\mathbb I}
\def\i{{\rm in}}\def\Sect{{\rm Sect}}  \def\H{\mathbb H}
\def\M{\scr M}\def\Q{\mathbb Q} \def\texto{\text{o}} \def\LL{\Lambda}
\def\Rank{{\rm Rank}} \def\B{\scr B} \def\i{{\rm i}} \def\HR{\hat{\R}^d}
\def\to{\rightarrow}\def\l{\ell}\def\iint{\int}
\def\EE{\scr E}\def\Cut{{\rm Cut}}\def\W{\mathbb W}
\def\A{\scr A} \def\Lip{{\rm Lip}}\def\S{\mathbb S}
\def\BB{\scr B}\def\Ent{{\rm Ent}} \def\i{{\rm i}}\def\itparallel{{\it\parallel}}
\def\g{{\mathbf g}}\def\Sect{{\scr Sec}}\def\T{\scr T}\def\V{{\bf V}}
\def\PP{{\bf P}}\def\HL{{\bf L}}\def\Id{{\rm Id}}\def\f{{\bf f}}\def\cut{{\rm cut}}

\maketitle

\begin{abstract}
Let  $\scr P$ be the space of  probability measures on $\R^d$. We associate a coupled nonlinear Fokker-Planck equation on $\R^d$, i.e. with solution paths in $\scr P$, to a linear  Fokker-Planck equation for probability measures on the product space $\R^d\times \scr P$, i.e. with solution paths in $\scr P(\R^d\times\scr P)$. We explicitly determine the corresponding linear Kolmogorov operator $\tilde{\mathbf{L}}_t$ using the natural tangent bundle over $\scr P$ with corresponding gradient operator $\nabla^\scr P$. Then it is proved that the diffusion process generated by $\tilde{\mathbf{L}}_t$ on $\R^d\times\scr P$ is intrinsically related to the solution of a McKean-Vlasov stochastic differential equation (SDE). We also characterize the ergodicity of the diffusion process generated by $\tilde{\mathbf{L}}_t$ in terms of asymptotic properties of the coupled nonlinear Fokker-Planck equation. Another main result of the paper is that the restricted well-posedness of the non-linear Fokker-Planck equation and its linearized version imply the (restricted) well-posedness of the McKean-Vlasov equation and that in this case the laws of the solutions have the Markov property.  All this is done under merely measurebility conditions on the coefficients in their measure dependence, hence in particular applies if the latter is of ``Nemytskii-type''. As a consequence, we obtain  the restricted weak well-posedness and the Markov property of the so-called nonlinear distorted Brownian motion, whose associated nonlinear Fokker-Planck equation is a porous media equation perturbed by a nonlinear transport term.  This realizes a programme put forward by McKean in his seminal paper of 1966 for a large class of nonlinear PDEs. As a further application we obtain a probabilistic representation of solutions  to Schr\"odinger type PDEs on $\R^d\times\scr P_2$, through the Feynman-Kac formula for the corresponding diffusion processes.
\end{abstract}

 \vskip 0.5cm

\section{Introduction}
As a first result of this paper (see Section 3) we identify the continuity equation corresponding to a non-linear Fokker-Planck equation on $\R^d$ with weakly continuous solution paths in $\scr P $, i.e. the space of probability measures on $\R^d$ equipped with the weak topology. We determine explicitly the vector field, defining the continuity equation, as a section in the natural tangent bundle over $\scr P $. More precisely, let for $m,d\in\N$
\begin{equation}\label{eq:1.1}
\begin{aligned}
b&=(b_i)_{1\leq i\leq d}\colon [0,\infty)\times\R^d\times\scr P \rightarrow \R^d,\\
\sigma&=(\sigma_{ij})_{1\leq i,j\leq d}\colon[0,\infty)\times\R^d\times\scr P \rightarrow \R^d\otimes \R^m
\end{aligned}
\end{equation}
be Borel measurable maps and consider the corresponding nonlinear Kolmogorov operator
\begin{equation}\label{eq:1.2}
\begin{aligned}
L_{t,\mu}h(x)&=\frac12\sum_{i,j=1}^d(\sigma\sigma^*)_{i,j}(t,x,\mu)\partial_i\partial_jh(x)+\sum_{i=1}^db_i(t,x,\mu)\partial_ih(x)\\
&=\frac12\big((\sigma\sigma^*)(t,x,\mu)\nabla\big)\cdot\nabla h(x)+b(t,x,\mu)\cdot\nabla h(x),
\end{aligned}
\end{equation}
where $(t,x,\mu)\in[0,\infty)\times\R^d\times\scr P $, $h\in C_0^2(\R^d)$, ``$\cdot$'' denotes inner product in $\R^d$ and $L_{t,\mu}$ determines the nonlinear Fokker-Planck equation
\begin{align}\label{eq:1.3}
\partial_t\mu_t=L^*_{t,\mu_t}\mu_t,
\end{align}
meant in the weak sense with test function space $C_0^\infty(\R^d)$ (see Definition \ref{44DF1*} below for details). By the recipe suggested in \cite{RWAKR}, \cite{AKR2} (see also \cite{RWAKR2}, \cite{RWAKR157}, \cite{MR}, \cite{RWORS}, \cite{RWOR}, \cite{RSAM} and \cite{RWR98}), which we briefly recall in the Appendix of this paper, one then finds the well-known tangent bundle over $\scr P $, namely
\begin{align}\label{eq:1.4}
(L^2(\R^d\rightarrow\R^d,\mu))_{\mu\in\scr P },
\end{align}
and a corresponding intrinsic gradient $\nabla^\scr P$ for a suitable and sufficiently large class $\scr FC_b^2(\scr P)$ of functions $F\colon\scr P \rightarrow\R^d$ (defined in \eqref{eq:2.4} below), which maps such $F$ into sections in this tangent bundle, i.e.
\begin{align}\label{eq:1.5}
\nabla^\scr PF(\mu)\in L^2(\R^d\rightarrow\R^d,\mu),\quad \mu\in\scr P
\end{align}
(for details see the Appendix).
$\nabla^\scr P$ is, of course, the well-known Otto-gradient introduced in \cite{Otto2001}. For the case where the space of $\mathbb{Z}_+$-valued Radon measures on $\R^d$ replaces $\scr P $, it was, however, already introduced in \cite{RWAKR}, \cite{AKR2}, even with a Riemannian manifold $M$ replacing $\R^d$. Then the continuity equation corresponding to the non-linear evolution equation \eqref{eq:1.3}, which is \textit{linear} and by definition an equation, whose solutions are weakly continuous paths of measure $\Gamma_t$, $t\geq0$, in $\scr P(\scr P )$ (the space of probability measures on $\scr P$ again equipped with the weak topology), is then given as
\begin{align}\label{eq:1.6}
\partial_t\Gamma_t=(\mathbf{L}_t)^*\Gamma_t
\end{align}
in the weak sense with test function space $\scr FC_b^2(\scr P)$ (see Section 3 for details). The corresponding Kolmogorov operator $\mathbf{L}_t$ is of first order and determined by a vector field in the tangent bundle $(L^2(\R^d\rightarrow\R^d,\mu))_{\mu\in\scr P }$ as follows:
\begin{align}\label{eq:1.7}
\mathbf{L}_t F(\mu)=\langle\frac12(\sigma\sigma^*)(t,\cdot,\mu)\nabla+b(t,\cdot,\mu),\nabla^\scr PF(\mu)\rangle_{L^2(\R^d\rightarrow\R^d,\mu)},
\end{align}
where $\mu\in\scr P $, $F\in\scr FC_b^2(\scr P)$ and $\langle\cdot,\cdot\rangle_{L^2(\R^d\rightarrow\R^d,\mu)}$ denotes the inner product in $L^2(\R^d\rightarrow\R^d,\mu)$ (see Proposition 3.1)  and for $\sigma\not\equiv 0$ it is meant in the usual generalized sense. So, every solution $\mu_t^\zeta$, $t\geq0$, to \eqref{eq:1.3} with initial condition $\zeta\in\scr P $ gives rise to a solution to \eqref{eq:1.6}  with initial condition the Dirac measure $\delta_\zeta$ in $\scr P(\scr P )$. Mixing these initial conditions $\zeta$ according a measure $\Gamma\in\scr P(\scr P )$, i.e., looking at the push forward measure under the flow generated by \eqref{eq:1.3} one obtains a solution path $\Gamma_t^\Gamma$, $t\geq0$, in $\scr P(\scr P )$ of \eqref{eq:1.6}. Interestingly, it turns out  that  the operator $\partial_t+\mathbf{L}_t$ with a suitable domain $\tilde{\scr F}$ is dissipative, hence closable on $L^1([0,T]\times\scr P ,\Gamma_t^\Gamma\; \d t)$. Thus $\partial_t+\mathbf{L}_t$ extends uniquely to a much larger space of functions on $[0,T]\times\scr P $ (including e.g. certain functions which are Lipschitz with respect to metrics generating the weak topology on $\scr P )$.Thus, through \eqref{eq:1.6} we can construct an abundance of measures of type $\Gamma_t^\Gamma\; \d t$ on $[0,T]\times\scr P $ for which the   first order operator $\partial_t+\mathbf{L}_t$, which is given by a vector field over $[0,T] \times \scr{P} (\mathbb{R}^d)$, i.e. a section in $( L^2(\R^d \rightarrow\R^d, \mu))_{\mu \in \scr{P}(\mathbb{R}^d)}$, and which in turn is canonically determined by the nonlinear Kolmogorov operator \eqref{eq:1.2}, is closable on $L^1 ([0,T] \times \scr{P}(\mathbb{R}^d); \Gamma_t^\Gamma\; \d t)$ (see Remark 3.2 below for details).

In the second main result of this paper (see Section 4.1) we obtain weak uniqueness in law and Markov properties for solutions to McKean-Vlasov equations from uniqueness of their corresponding non-linear Fokker-Planck equation and their $``$freezed" linear version. More precisely, for $b$ and $\sigma$ as in \eqref{eq:1.1} consider the McKean-Vlasov stochastic differential equation (SDE) on $\R^d$ (see \cite{RWCD} and the references therein)
\begin{align}\label{eq:1.8}
\d X_t=b(t,X_t,\scr L_{X_t})\d t+\sigma(t,X_t,\scr L_{X_t})\d W_t,\quad t\geq0,
\end{align}
where $W_t$, $t\geq0$, is an $(\scr F_t)$-Brownian motion in $\R^m$ defined on a probability space $(\Omega, \scr F, \mathbb{P})$ with normal filtration $(\scr F_t)_{t\geq0}$ and the solution $X_t$, $t\geq0$, is an $(\scr F_t)$-adapted stochastic process on $(\Omega,\scr F)$ with $\mathbb{P}$-a.s. continuous sample paths in $\R^d$ and time marginal laws $\scr L_{X_t}:=\mathbb{P}\circ X_t^{-1}$, $t\geq0$.
Then obviously by It\^o's formula for any (probabilistically) weak solution to \eqref{eq:1.8} its time marginals $\mu_t:=\scr L_{X_t}$, $t\geq0$, solve the nonlinear Fokker-Planck equation \eqref{eq:1.3}. Recently, it was proved in [5], [6] that under a natural integrability condition on $\sigma$ and $b$ the converse is also true. Hence in this sense weak solutions to McKean-Vlasov SDE and nonlinear Fokker-Planck equations are equivalent. Once one has this equivalence, it is fairly straightforward to get a sufficient condition for weak uniqueness for \eqref{eq:1.8}. For this we consider the $``$freezed" linear version of the non-linear Fokker-Planck equation \eqref{eq:1.3}, i.e. for a fixed solution $\mu_t$, $t\geq0$, of \eqref{eq:1.3}
\begin{align}\label{eq:1.9}
\partial_t\nu_t=L^*_{t,\mu_t}\nu_t
\end{align}
and look at the pair of Fokker-Planck equations
\begin{align}\label{eq:1.10}\begin{cases}
\partial_t\mu_t&=L^*_{t,\mu_t}\mu_t\\
\partial_t\nu_t&=L^*_{t,\mu_t}\nu_t.
\end{cases}
\end{align}
Then we introduce a notion of $``$restricted"  well-posedness for \eqref{eq:1.10} (see Definition \ref{44DF1*}(2) below). It is $``$restricted"  in the sense that we restrict to subclasses of probability measures as initial data and to subclasses of solutions with certain properties. This restriction causes major technical complications, but is necessary, because  in many applications \eqref{eq:1.10} is only well posed  on a subspace of $\scr P$.  In particular, as far as  the distribution density function is concerned ,  it is natural to consider the class of    absolutely continuous probability measures ( see Section 5).  Furthermore, when $b(t,x,\cdot)$ and $\si(t,x,\cdot)$ are Lipschitz continuous in the Wasserstein distance $\mathbb W_2$, \eqref{eq:1.8}  and thus \eqref{eq:1.3} is usually  solved for initial distributions having finite second moment, see \cite{RWPECR19, RW18HW, RWW18} and references within.   Theorem \ref{prp41} below provides the equivalence of the restricted well-posedness of \eqref{eq:1.10} with that of the corresponding SDE-version, i.e.,
\begin{align}\label{eq:1.11}
\begin{cases}
\partial_t\mu_t &= L_{t,\mu_t}^*\mu_t\\
\d X(t)&=b(t,X(t),\mu_t)\d t+\sigma(t,Y(t),\mu_t)\d W(t).
\end{cases}
\end{align}
From Theorem \ref{prp41} we then deduce that the restricted well-posedness of \eqref{eq:1.10} implies that of the McKean-Vlasov SDE \eqref{eq:1.8}. Under the assumption of restricted well-posedness we also prove that the laws of the solutions to both the second equation in \eqref{eq:1.11} and to \eqref{eq:1.8} have the Markov property (see Theorem 4.5 and Corollary 4.6 in Section 4.3). The proof, however, is much more involved in comparison to the case of well-posedness for all initial conditions.
It turns out that the same results hold, if in the second equations of \eqref{eq:1.10} and \eqref{eq:1.11}, we change the operator $L_{t,\mu_t}$, i.e. the coefficients $b$ and $\sigma$, to another Kolmogorov operator $\bar{L}_{t,\mu_t}$ with coefficients $\bar{b},\bar{\sigma}$. Therefore, we formulate our results in this more general case.

In Section 5 we apply these results to an interesting example, where the above nonlinear Fokker-Planck equation is a porous medium equation perturbed by a first order (transport) term recently studied in \cite{RWBR19}. The solution to the corresponding McKean-Vlasov SDE can be considered as a nonlinear distorted Brownian motion (see \cite{RWBR19} for details).

At this point we would like to stress that for these applications it is important that our above framework works for coefficients $b$ and $\sigma$, which are assumed to be only measurable in the measure variable $\mu$.
This implies that we can cover cases, where the nonlinear Fokker-Planck equation \eqref{eq:1.3} is a nonlinear PDE, because in such cases $b$ and $\sigma$ depend ``Nemytskii-type'' (hence very singularly on $\mu$), i.e. $b(t,x,\cdot)$ and $\sigma(t,x,\cdot)$ depend on the measure through its Radon-Nikodym density $\rho$ (w.r.t. Lebesgue measure) evaluated at $x$ (see Section 5 for details).
The main result here is Theorem \ref{thm51}, by which in particular we have the Markov property for the solutions of the corresponding McKean-Vlasov SDE (see \eqref{EQ58}).
Consequently, the results of Section 5 can be considered as a realization of McKean's vision from \cite{McK66} (namely to associate certain nonlinear PDEs with non-linear Markov processes analogously to the well-established linear case) for a large class of non-linear PDEs, more precisely, for nondegenerate porous media equations perturbed by a transport term.

The third main result of this paper is about transforming the non-linear coupled Fokker-Planck equation \eqref{eq:1.10} into a \textit{linear} Fokker-Planck equation on $\R^d\times\scr P $, i.e. with solution paths in $\scr P(\R^d\times\scr P )$ (see Section 4.2). The motivation comes from the hope that for the understanding of McKean-Vlasov SDEs it could be useful to study the pair process $(X(t),\scr L_{X(t)})$, $t\geq0$, on the state space $\R^d\times\scr P $. By the Appendix of this paper we know that the corresponding tangent bundle is
\begin{align}\label{eq:1.12}
\R^d\oplus(L^2(\R^d\rightarrow\R^d;\mu))_{\mu\in\scr P }.
\end{align}
As a consequence of this and Section 3  it is trivial to  derive the Fokker-Planck equation associated to the process $(X(t),\scr L_{X(t)})_{t\geq0}$ on $\R^d\times\scr P $ which  again is  a \textit{linear} Fokker-Planck equation, namely (see Definition \ref{44DF1*} below)
\begin{align}\label{eq:1.13}
\partial_t\Lambda_t=\tilde{\mathbf{L}}^*_t\Lambda_t,\quad t\geq0,
\end{align}
where $(\Lambda_t)_{t\geq0}$ is a weakly continuous path of probability measures on $\R^d\times\scr P $, i.e. in $\scr P(\R^d\times\scr P )$, and with the corresponding \textit{linear} Kolmogorov operator $\tilde{\mathbf{L}}_t$ (of course, first order in $\mu$) on $\R^d\times\scr P $ being given explicitly on a reasonably rich class $\scr C$ of functions $G:\R^d\times\scr P \rightarrow\R$ as
\begin{align}\label{eq:1.14}
\tilde{\mathbf{L}}_t G=\bar{\mathbf{L}}_t G+\mathbf{L}_t G
\end{align}
where ${\mathbf{L}}_t$ is as in \eqref{eq:1.7} and
\begin{align}\label{eq:1.15}
\bar{\mathbf{L}}_t G(x,\mu)= L_{t,\mu}(G(\cdot,\mu))(x)
\end{align}
(see Section 4.2 for details).
The exact relation between solutions to \eqref{eq:1.10} and \eqref{eq:1.13} is given in Theorem \ref{thm43}(i) below. In particular, an explicit formula for solutions of \eqref{eq:1.13} is given through probability kernels $\mathbf{P}_{s,t}(x,\zeta;dy,d\nu)\in\scr P(\R^d\times\scr P )$, $(x,\zeta\in\R^d\times\scr P $, $s\leq t$, which satisfy the Chapman-Kolmogorov equations (see Theorem \ref{thm43}(ii)).

As another consequence we prove the Markov property of the law of the solution $(X(t), \scr{L}_{X(t)})$ of the McKean-Vlasov SDE \eqref{eq:1.8} in a stronger form under the stronger condition \eqref{EQ432} (see Theorem 4.11 in Section 4.4 below for the precise formulation of this result).\\

In the time homogeneous case a further application of the results in Section 4 is a characterization of the ergodicity for solutions to \eqref{eq:1.11} in terms of the ergodicity of $(\mathbf{P}_{s,t})_{s\leq t}$ (see Theorem 4.12 in Section 4.5). An application of this result is presented in Section 6 and concerns a case with more regular coefficients $b$ and $\sigma$ (see  Theorem \ref{44T4.2} below) where $\scr{P} (\mathbb{R}^d)$ is replaced by $\scr{P}_2(\mathbb{R}^d)$, i.e. all elements in $\scr{P}(\mathbb{R}^d)$ with finite second moments (equipped with the Wasserstein metric). In this case even exponential ergodicity is proved.

As a further consequence, in Section 7 we then prove a Feynman-Kac formula for the above diffusion process on $\mathbb{R}^d \times \scr{P}(\mathbb{R}^d)$, from which we derive a probabilistic representation for solutions of Schr\"odinger type PDE on $\mathbb{R}^d \times \scr{P}_2(\mathbb{R}^d)$ of the following form
\begin{equation}\label{44PDE}
\pp_t u(t,x,\zeta) + {\tilde{\mathbf{L}}_t} u(t,\cdot,\cdot)(x,\zeta) + (\V  u)(t,x,\zeta) + \f(t,x,\zeta)=0,\ \ t\in [0,T],
\end{equation}
where $T>0$ is fixed, $(x,\zeta)\in  \R^d\times \scr P_2(\mathbb{R}^d),$ and $\V,\f$ are measurable functions on $[0,T]\times\R^d\times\scr P_2(\mathbb{R}^d)$.
This generalizes some known results from the literature (see \cite{RWLP,RWChass,RWCM,RWHL,RWLJ}).\\

Finally we would like to emphasize again that the main motivation of this paper is to contribute to the theory of nonlinear Fokker-Planck equations on one side and McKean-Vlasov SDEs on the other. The literature on both parts of the theory is overwhelming, so that we apologize that an overview of the known results is beyond the scope of this paper. Therefore, we confine ourselves to refer the reader to the monographs \cite{F2005} and \cite{RWBKRS15} concerning Fokker-Planck(-Kolmogorov) equations and to \cite{RWCD,K10}  concerning McKean-Vlasov SDEs as well as the references therein and e.g. the very recent papers \cite{RWBR18,RWBR,RWCard,RWPECR19,RWHSS,RWHL,RW18HW,RWMV,RWRW18b,RWRW18,RWRW19b,RWW18}. Furthermore, we would like to stress that according to general Markov process theory via the martingale problem given by the underlying generator (see e.g. \cite{RWSV} in the classical case) there is a close connection of our results with the very nice recent works (\cite{RWChass}, \cite{F2008}) on the It\^o-formula for the process $(X_t, \scr{L}_{X(t)}),\ t\geq 0$, coming from the McKean-Vlasov SDE \eqref{eq:1.8} above. The connection is obvious, since one can show that our intrinsic gradient $\nabla^\scr{P}$ on functions in $\scr{F}C^2_b(\scr{P})$ (see Appendix A) is the same as the Lions-derivative from \cite{RWCard}. This fact was proved in \cite{RWRW19b}. Since, however, our approach is more analytic and based more on nonlinear Fokker-Planck equations, we do not need this It\^o-formula.

\section{Preliminaries and notation}
Let $\scr P$ denote the set of all probability measures on $\R^d$ equipped with the weak (= narrow) topology and corresponding Borel-$\sigma$-algebra $\scr B(\scr P)$. Likewise $\scr P(\scr P)$ and $\scr P(\R^d\times\scr P)$ denote the set of all probability measures on $\scr P$ and $\R^d\times\scr P$ respectively and both are considered with the weak topology and corresponding Borel-$\sigma$-algebras.
Let \beq\label{eq:2.0}
b,\bar b: [0,\infty)\times \R^d\times\scr P\to \R^d;\ \ \si,\bar\si: [0,\infty)\times \R^d\times \scr P\to \R^d\otimes \R^{m}, 
\end{equation}
be Borel-measurable maps, where $m,d\in\N$. Furthermore, besides $L_{t,\mu}$ in \eqref{eq:1.2} we define the following measure-dependent Kolmogorov operator  on $\R^d$:
\begin{align}
\bar L_{t,\mu}h(x):= \ff 1 2 \sum_{i,j=1}^d (\bar\si\bar\si^*)_{ij} (t,x,\mu)\pp_i\pp_j h(x) +\sum_{i=1}^d \bar b_i(t,x,\mu) \pp_i h(x),\ \ h\in C^2_0(\R^d),\label{eq:2.2}
\end{align}
 and consider (more generally than \eqref{eq:1.10}) the coupled non-linear Fokker-Planck equations
\begin{align}\label{eq:2.2'}
\begin{split}
\begin{cases}
\partial_t\mu_t = L^*_{t,\mu_t}\mu_t,\\
\partial_t\nu_t = \bar{L}^*_{t,\mu_t}\nu_t.
\end{cases}
\end{split}
\end{align}
Let us define the following test function spaces:
\begin{equation}\label{44FC}
\F C_b^2(\scr P):= \big\{F(\mu):=g(\mu(h_1),\cdots, \mu(h_n)):\ n\ge 1,g\in C_b^1(\R^n),  h_i\in C_0^2(\R^d)\big\}.
\end{equation}
on $\scr P$ and
\begin{equation}\label{eq:2.4}
\scr C:=\{(x,\mu)\mapsto h_0(x)F(\mu):h_0\in C_0^2(\R^d), F\in\scr FC_b^2(\scr P) \}.
\end{equation}
Next, consider the time-dependent  differential  operator $\tilde{\mathbf{L}}_t$ defined by \eqref{eq:1.14}  and  \eqref{eq:1.15}  with $\bar{L}_{t,\mu}$ replacing $L_{t,\mu}$.
\eqref{eq:2.2'} is  meant in the weak sense (see Definition \ref{44DF1*}(2) below). In Section 4 we are going to establish a correspondence between solutions $(\mu_t,\nu_t)_{t\geq0}$ of \eqref{eq:2.2'}  and those to the   \textit{linear} Fokker-Planck equation
\begin{align}\label{eq:2.2''}
\partial_t\Lambda_t=\tilde{\mathbf{L}}^*_t\Lambda_t,\ t\geq0,
\end{align}
on $\R^d\times \scr P$ 
again meant in the weak sense (see Definition \ref{44DF1*}(3) below). 
\eqref{eq:2.2''} is a generalization of \eqref{eq:1.13} with $\bar{L}_{t,\mu}$ replacing $L_{t,\mu}$ in \eqref{eq:1.15}. 
Let $C([s,\infty)\rightarrow\scr P)$ denote the set of all weakly continuous paths in $\scr P$ starting from $s\in[0,\infty)$ and $C([s,\infty)\rightarrow\scr P(\scr P))$, $C([s,\infty)\rightarrow\scr P(\R^d\times\scr P))$are defined likewise.
\begin{rem}
For nonlinear Fokker-Planck equations \eqref{eq:1.3} typically one cannot expect to have a unique solution for every initial condition $\zeta\in\scr P$ at time $s\in[0,\infty)$, but only for $\zeta\in\scr P_0$, where $\scr P_0\in\scr B({\scr P})$. In addition, even for such restricted initial conditions $\zeta\in\scr P_0$ generally one does not have a unique solution to \eqref{eq:1.3} in all of $C([s,\infty)\rightarrow\scr P)$, but rather in a subset thereof,  whose paths in particular leave $\scr P_0$ invariant. Therefore, we introduce property  {\bf (P)}  below.
\end{rem}
For $\scr P_0\in\scr B(\scr P)$ and $\scr A\subset C([0,\infty)\rightarrow\scr P_0)$ we need the following property for the pair $(\scr P_0,\scr A)$:
\begin{enumerate}
\item[{\bf (P)}] If $\tilde{\mu}\in\scr P$ such that $\tilde{\mu}\leq C\mu$ for some $\mu\in\scr P_0$, $C\in(0,\infty)$, then $\tilde{\mu}\in\scr P_0$. If $(\lambda_t)_{t\geq0}\in\scr A$, then $(\lambda_{s+t})_{t\geq0}\in\scr A$ for all $s\geq0$, and $(\tilde{\lambda}_t)_{t\geq0}\in\scr A$, provided $(\tilde{\lambda}_t)_{t\geq0}\in C([0,\infty)\rightarrow\scr P_0)$ and $\tilde{\lambda}_t\leq C\lambda_t$ for all $t\geq0$ and some $C\in(0,\infty)$.
\end{enumerate}
\begin{defn}[Solution to Fokker-Planck equations, see \cite{RWBKRS15}] \label{44DF1*} \
Let $s\ge 0$ be fixed.
\begin{enumerate}
\item[$(1)$] $(\mu_{t})_{t\geq s}\in C([s,\infty)\to \scr P)$ is called a solution to
\eqref{eq:1.3}
from time $s$, if for all $t\in[s,\infty)$
\begin{equation}\label{44A1**}
\int_s^t \d r \int_{\R^d}  \big(|b|+ \|\si\|^2\big) (r,x,\mu_r)\mu_r(\d x) <\infty,
\end{equation}
and
\begin{equation}\label{44A2*}
\int_{\R^d} h\d\mu_t= \int_{\R^d} h\d\mu_s +\int_s^t\d r \int_{\R^d} L_{r,\mu_r}h \,\d\mu_r,\ \ h\in C_0^\infty(\R^d).
\end{equation}
\item[$(2)$] A pair $(\mu_t,\nu_t)_{t\geq s}$ with $(\mu_t)_{t\geq s}$, $(\nu_t)_{t\geq s}\in C([s,\infty)\to \scr P)$ is called a solution to \eqref{eq:2.2'}  from time $s$, if $(\mu_t)_{t\geq s}$ is a solution of \eqref{eq:1.3} from time $s$ and for all $t\in[s,\infty)$
	\begin{align}\label{44A1}
	\int_s^t\d r\int_{\R^d}    \big(|\bar b|+ \|\bar \si \|^2\big)(r,x,\mu_r)\nu_r(\d x)<\infty
	\end{align} and
	\begin{equation}\label{44A2}
	\int_{\R^d} h\d\nu_t= \int_{\R^d} h\d\nu_s +\int_s^t\d r \int_{\R^d} \bar L_{r,\mu_r}h \,\d\nu_r,\ \ h\in C_0^\infty(\R^d).
	\end{equation}
	For two pairs $(\scr P_0,\scr A),(\tilde{\scr P}_0,\tilde{\scr A})$ with property {\bf (P)} we call \eqref{eq:2.2'}  well-posed in \\
	$\left((\scr P_0,\scr A),(\tilde{\scr P}_0,\tilde{\scr A})\right)$ if the following holds:
	\begin{enumerate}
		\item[$($a$)$] For every $(s,\zeta)\in[0,\infty)\times\scr P_0$ there exists a unique solution $(\mu_{s,t}^\zeta)_{ t\geq s}$ to \eqref{eq:1.3} starting from $s$ with $\mu_{s,s}^\zeta = \zeta$ and  $(\mu_{s,s+t}^\zeta)_{t\geq0}\in\scr A$ such that $\zeta\mapsto \mu_{s,t}^\zeta$ is Borel measurable for all $ t\geq s$.
		\item[$($b$)$] For every $(s,\theta,\zeta)\in[0,\infty)\times\tilde{\scr P}_0\times\scr P_0$ and $(\mu_{s,t}^\zeta)_{ t\geq s}$ as in (a) there exists a unique solution $(\nu_{s,t}^{\zeta,\theta})_{t\geq s}$ to \eqref{44A2} with $\mu_{s,r}^\zeta$ replacing $\mu_r$, $r\geq s$, starting from $s$ with $\nu_{s,s}^{\zeta,\theta}=\theta$ and $(\nu_{s,s+t}^{\zeta,\theta})_{t\geq 0}\in\tilde{\scr A}$ such that $(\theta,\zeta)\mapsto\nu_{s,t}^{\zeta,\theta}$ is Borel measurable for all $ t\geq s$.
	\end{enumerate}
\item[$(3)$] $(\LL_t)_{t\geq s}\in C([s,\infty)\to \scr P(\R^d\times\scr P))$ is called a solution to \eqref{eq:2.2''}  from time $s$, if for all $t\in[s,\infty)$
	\begin{equation}\label{44A1*}
    \int_s^t \d r \int_{\R^d}  \big(\big|(|b|+\|\sigma\|^2)(r,\cdot,\mu)\big|_{L^1(\R^d,\mu)}+\left(|\bar b| +\|\bar\si\|^2\big)(r,x,\mu)\right) \LL_r(\d x,\d\mu) <\infty,
	\end{equation}
	and for any $G\in \C,$
	\begin{equation}\label{44A2**}
 	\int_{\R^d\times\scr P} G\d\LL_t= \int_{\R^d\times\scr P} G\d\LL_s +\int_s^t\d r \int_{\R^d\times\scr P} \tilde{\mathbf{L}}_{r} G\,\d\LL_r.
	\end{equation}
\item[$(4)$] $(\Gamma_{t})_{t\geq s}\in C([s,\infty)\rightarrow\scr P(\scr P ))$ is called a solution to \eqref{eq:1.6} from time s, if for all $t\in[s,\infty)$
	\begin{align}\label{eq:1.24}
	\int_s^t\ \d r\int_{\R^d}\big|(|b|+\|\sigma\|^2)(r,\cdot,\mu)\big|_{L^1(\R^d,\mu)}\ \Gamma_r(\d\mu)<\infty,
	\end{align}
	and for any $F\in\scr FC_b^2(\scr P)$
	\begin{align}\label{eq:1.25}
	\int_\scr PF\ \d\Gamma_t=\int_\scr PF\ \d\Gamma_s+\int_s^t\d r\int_\scr P\mathbf{L}_r F\ \d\Gamma_r.
	\end{align}
\end{enumerate} \end{defn}
\begin{rem}\label{rem22}
If \eqref{eq:1.10} is well-posed, $b=\bar{b},\ \sigma=\bar{\sigma}$ and $\scr P_0\subset\tilde{\scr P_0}$, we have for all $(s,\zeta)\in[0,\infty)\times\scr P_0$
\begin{align*}
\nu_{s,t}^{\zeta,\zeta}=\mu_{s,t}^{\zeta},\ t\geq 0
\end{align*}
because both solve \eqref{eq:1.3} with the same initial condition $\zeta$.
\end{rem}
Let us consider another coupled equation involving the first equation of \eqref{eq:2.2'}  and the stochastic equation corresponding to the linear Kolmogorov operator $\bar{L}_{\mu,t}$ in \eqref{eq:1.2}, i.e.
\begin{equation}\label{44EE}
\begin{cases}
\pp_t \mu_t= L_{t,\mu_t}^* \mu_t,\\
\d X_t= \bar b(t, X_t,   \mu_t)\d t + \bar \si(t,  X_t,  \mu_t)\d W_t,
\end{cases}
\end{equation}
with solution paths $(X_t,\mu_t)_{t\geq0}$ in $\R^d\times\scr P$. Here $(W_t)_{t\geq0}$ is an $(\scr F_t)$-Brownian motion on a probability space $(\OO,\F, \P)$ with normal filtration $(\scr F_t)_{t\geq0}$ and values in $\R^m$. Below we set $\scr L_{X_t}:=\mathbb{P}\circ X_t^{-1}$ ($``$time marginal law" at $t$).
We now introduce the notion of weak solution and weak well-posedness for \eqref{44EE}.
\begin{defn}\label{44DF1}
\begin{enumerate}
\item[$($i$)$] Let $(s,\theta,\zeta)\in [0,\infty)\times\scr P\times\scr P$. If \eqref{eq:1.3} has a  solution $(\mu_{t})_{t\ge s} $ with $\mu_{s}=\zeta$, and the SDE
\begin{equation}\label{44E2}
\d X_{s,t}^{\zeta,\theta}= \bar b(t,X_{s,t}^{\zeta,\theta}, \mu_t)\d t+ \bar\si(t,X_{s,t}^{\zeta,\theta}, \mu_t)\d W_t,\ \ t\ge s, \scr L_{X_{s,s}^{\zeta,\theta}}=\theta,
\end{equation}
has a pathwise continuous weak solution, then $(X_{s,t}^{\zeta,\theta},\mu_{t})_{t\ge s}$ is called a weak solution to \eqref{44EE} with initial value $(\theta,\zeta)$ at time $s$.
\item[$($ii$)$] Let $(\scr P_0,\scr A)$, $(\tilde{\scr P_0},\tilde{\scr A})$ be two pairs with property {\bf (P)}. We call \eqref{44EE} well-posed in $\left((\scr P_0,\scr A),(\tilde{\scr P_0},\tilde{\scr A})\right)$, if (a) in Definition \ref{44DF1*} (2) holds and if for every $(s,\theta,\zeta)\in[0,\infty)\times\tilde{\scr P_0}\times\scr P_0$ and for the unique solution $(\mu_{s,t}^\zeta)_{s\geq t}$ to \eqref{eq:1.3} there exists a unique weak solution $(X_{s,t}^{\zeta,\theta})_{t\geq s}$ to \eqref{44E2} with $\mu_{s,t}^\zeta$ replacing $\mu_t$, $t\geq s$, such that for all $ t\geq s$ the law of $X_{s,t}^{\zeta,\theta}$ is Borel-measurable in $(\zeta,\theta)$.
\end{enumerate}
\end{defn}

\begin{rem}\label{rem24}
Obviously, in the situation of Definition \ref{44DF1}(i) we have by It\^{o}'s formula that $\nu_{s,t}^{\zeta,\theta}:=\scr L_{X_{s,t}^{\zeta,\theta}}$, $t\geq s$, is a solution to \eqref{44A2}.
\end{rem}
Now we consider the McKean-Vlasov SDE \eqref{eq:1.8}

\begin{defn}\label{D3.3}
\begin{enumerate}
\item[(i)] Let $(s,\zeta)\in[0,\infty)\times\scr P$. The McKean-Vlasov SDE \eqref{eq:1.8} is said to have a weak solution starting from $s$ with time marginal law $\zeta$, if there exist  an $(\scr F_t)$-Brownian motion $W_t, t\ge 0,$ on a probability space $(\OO,\F,\P)$ with normal filtration $(\scr F_t)_{t\geq0}$ and values in $\R^m$ and an $(\F_t)$-adapted continuous process $(X_t)_{t\ge s}$ in $\R^d$, such that $\L_{X_s}=\zeta$, $$\int_s^t \d r\int_{\R^d} \big(|b|+\|\si\|^2\big)(r,x,\L_{X_r})\L_{X_r}(\d x)<\infty,\ \ t\ge s,$$  and $\P$-a.s.
$$ X_t= X_s+\int_s^t b(r,X_r,\L_{X_r})\d r + \int_s^t \si(r,X_r,\L_{X_r}) \d W_r,\ \ t\ge s.$$
\item[(ii)] Let $(\scr P_0,\scr A)$ be a pair with property {\bf (P)}. We call \eqref{eq:1.8} weakly well-posed in $(\scr P_0,\scr A)$ if for every $(s,\zeta)\in[0,\infty)\times\scr P_0$ and any two weak solutions with time marginal laws $(\scr L_{X_t})_{t\geq s}$, $(\scr L_{\tilde{X}_t})_{t\geq s}\in \scr A$ and $\scr L_{X_s}=\scr L_{\tilde{X}_s}=\zeta$, we have that their laws coincide.
\end{enumerate}
\end{defn}
We close this section with recalling the following result from \cite{RWBR18,RWBR}.
\begin{thm}\label{thm24}
Let $(s,\zeta)\in[0,\infty)\times\scr P$. Then the McKean-Vlasov SDE \eqref{eq:1.8} has a weak solution starting from $s$ with time marginal law $\zeta$, if and only if \eqref{eq:1.3} has a solution $(\mu_t)_{t\geq s}$ starting from $s$ with $\mu_s=\zeta$. In this case $\mu_t=\scr L_{X_t}$, $t\geq s$.
\end{thm}
\begin{proof}
If \eqref{eq:1.8} has a weak solution starting from $s$ with time marginal law $\zeta$ then by Remark \ref{rem24} it follows that $\mu_t:=\scr L_{X_t}$, $t\geq s$, is a solution of \eqref{eq:1.3} starting from $s$ with $\mu_s=\zeta$. The converse is proved in Section 2 of \cite{RWBR} (and \cite{RWBR18}).
\end{proof}

\section{Linearization of nonlinear Fokker-Planck equations and construction of probability measures on $[s,\infty)\times\scr P $}
Let $s\geq 0$ be fixed. The nonlinear Fokker-Planck equation \eqref{eq:1.3} can be considered as an evolution equation in $\scr P $. For evolution equations there is a standard way to linearize them by transforming them to an evolution equation on the space of probability measures over their state space, hence in our case onto $\scr P(\scr P )$, i.e. the space of probability measures on $\scr P $ equipped with the Borel-$\sigma$-algebra generated by the weak topology on $\scr P $. More precisely, for weakly continuous solution paths $(\mu_t)_{t\geq s}$ for \eqref{eq:1.3} in $\scr P $, one derives an equation for the paths $(\delta_{\mu_t})_{t\geq s}$ in $\scr P(\scr P )$ as follows:

For all $F\in\scr FC^2_b(\scr P)$,
\begin{align}\label{eq:3.2}
F(\mu)=f(\mu(h_1),\dots,\mu(h_n)),\ n\in \N,\ h_1,\dots,h_n\in C^2_0(\R^d),\ f\in C_b^1(\R^n),
\end{align}
 and  for any  solution  path $(\mu_t)_{t\geq s}$  to \eqref{eq:1.3}  in $\scr P $, by the chain rule and \eqref{eq:A.22} in the Appendix  we have 
\begin{align}\label{eq:3.3}
\frac{\d}{\d t}\delta_{\mu_t}(F)&=\frac{\d}{\d t}F(\mu_t)\notag
=\sum_{i=1}^n\partial_if(\mu_t(h_1),\dots,\mu_t(h_n))\partial_t\mu_t(h_i)\notag\\
&=\sum_{i=1}^n\partial_if(\mu_t(h_1),\dots,\mu_t(h_n))\int_{\R^d}L_{t,\mu_t}h_id\mu_t\notag\\
&=\sum_{i=1}^n\Big(\partial_if(\mu_t(h_1),\dots,\mu_t(h_n))\notag\\
&\qquad\cdot\int_{\R^d}\big(\frac12(\sigma\sigma^*)(t,x,\mu_t)\nabla\cdot\nabla h_i(x)+b(t,x,\mu_t)\cdot\nabla h_i(x)\big)\mu_t(dx)\Big)\\
&=\delta_{\mu_t}\big(\langle\frac12(\sigma\sigma^*)(t,\cdot)\nabla+b(t,\cdot),\nabla^\scr PF(\cdot)\rangle_{L^2(\R^d\rightarrow\R^d,\cdot)}\big),\notag
\end{align}
where $(\sigma\sigma^*)(t,\mu_t)$, $b(t,\mu_t)$ denote the maps
\begin{align*}
\R^d\ni x\rightarrow\sigma\sigma^*(t,x,\mu_t),\\
\R^d\ni x\mapsto b(t,x,\mu_t),
\end{align*}
which by assumption \eqref{44A1**} (as part of the definition of the solution to \eqref{eq:1.3}) are $\mu_t$-integrable. Furthermore, here we set $\delta_{\mu_t}(g(\cdot)):=\int g(\nu)\delta_{\mu_t}(d\nu)=g(\mu_t)$ for a Borel measurable map $g\colon\scr P\rightarrow\R$. Hence we obtain
\begin{align}
\partial_t\delta_{\mu_t}(F)=\delta_{\mu_t}(\langle\frac12(\sigma\sigma^*)(t,\cdot)\nabla+b(t,\cdot),\nabla^\scr PF(\cdot)\rangle_{L^2(\R^d\rightarrow\R^d,\cdot)})\quad&\text{ for all $F\in\scr FC_b^2(\scr P)$,}\label{eq:3.4}
\end{align}
which clearly is a linear equation for $\delta_{\mu_t}$, $t\geq s$, in $\scr M(\scr P )$, i.e. the space of all bounded variation measures on $\scr P $. Vice versa, by the above derivation, \eqref{eq:3.4} implies \eqref{eq:1.3} by just taking $F_l\in\scr FC_b^2(\scr P)$, $l\in\N$, such that for $l\in\N$
\begin{align*}
F(\mu)=f_l(\mu(h)),\quad h\in C_b^2(\R^d),
\end{align*}
and $f_l\in C_b^1(\R^1)$ such that $\frac{\d}{\d x}f_l\rightarrow1$ as $l\rightarrow\infty$. Hence we have proved the following result.
\begin{prp}
A weakly continuous $\scr P $-valued path $(\mu_t)_{t\geq s}$ satisfies the nonlinear F-P. equation \eqref{eq:1.3} in the sense of Definition \ref{44DF1*}(1), if and only if the $\scr P $-valued path $(\delta_{\mu_t})_{t\geq s}$ satisfies the linear first order F-P equation \eqref{eq:3.4}.
\end{prp}
Now we proceed in the canonical way: For $\zeta\in\scr P $ let us denote a solution $(\mu_t)_{t\geq s}$ to \eqref{eq:1.3} with initial condition $\mu_s=\zeta$ by $\mu(t,\zeta)$, $t\geq s$. Then consider $\scr P $ equipped with the $\sigma$-algebra $\tilde{\scr{B}}$ generated by the maps
\begin{align*}
\scr P \ni\zeta\mapsto\mu(t,\zeta)\in\scr P ,\ t\geq s,
\end{align*}
where the image space is considered with the Borel $\sigma$-algebra $\scr B(\scr P)$. Then for any probability measure $\Gamma$ on $(\scr P ,\tilde{\scr B})$ define
\begin{align*}
\Gamma_t:=\Gamma\circ\mu(t,\cdot)^{-1}.
\end{align*}
Then by \eqref{eq:3.3} for all $F\in\scr FC_b^2(\scr P)$
\begin{equation}\label{eq:3.5}\begin{aligned}
&\frac{\d}{\d t}\int F(\mu)\ \Gamma_t(d\mu)\\
&=\int \frac{\d}{\d t}F(\mu(t,\zeta))\ \Gamma(d\zeta)\\
&=\int\langle\frac12(\sigma\sigma^*)(t,\mu)\nabla+b(t,\mu),\nabla^{\scr P}F(\mu)\rangle_{L^2(\R^d\rightarrow\R^d,\mu)}\ \Gamma_t(\d\mu),
\end{aligned}\end{equation}
and rewriting this in the weak sense (with test function space $\scr FC_b^2(\scr P)$, see Definition \ref{44DF1*}(4)) we obtain
\begin{align}\label{eq:3.6}
\frac{\d}{\d t}\Gamma_t=\mathbf{L}_t\Gamma_t,\ \ \Gamma_0=\Gamma,
\end{align}
where $\mathbf{L}_t$ is defined in \eqref{eq:1.7}.
Whether the above $\sigma$-algebra $\tilde{\scr B}$ on $\scr P $ coincides with $\scr B(\scr P)$ has to be checked in every particular case and is, of course, fulfilled if the solution to \eqref{eq:1.2} is continuous in its initial condition $\zeta$ with respect to a suitable topology, as is e.g. the case for our main examples below.
\begin{rem}\label{rem32}
By the product rule \eqref{eq:3.6} implies that for every $\Gamma$ as above, $T\in(0,\infty)$ and for all $F\in\scr FC_b^2(\scr P)$, $G\in C^1([0,T];\R^d)$ with $G(T)=0$, and $\tilde F(t,\mu):=G(t)F(\mu)$, $(t,\mu)\in[0,T]\times\scr P$
\begin{align*}
\int_0^T\int_\scr P (\frac{\partial}{\partial t}+\mathbf{L}_t)\tilde{F}\ \d \Gamma_t\ \d t =-\int \tilde F(0,\mu)\ \Gamma(\d\mu).
\end{align*}
Letting $\tilde{\scr F}$ denote the linear space of all such functions $\tilde F$, we obtain that for all nonnegative $\tilde{F}\in\tilde{\scr F}$
\begin{align*}
\int_0^T\int_\scr P (\frac{\partial}{\partial t}+\mathbf{L}_t)\tilde F \ \d \Gamma_t\ \d t\leq 0,
\end{align*}
i.e. the operator $\frac{\partial}{\partial t}+\mathbf{L}_t$ with domain $\tilde{\scr F}$ is dissipative on $L^1([0,T]\times\scr P , \Gamma_t\d t)$, hence in particular closable. So, by the above we have a means to construct a whole class of  finite nonnegative  measures on $[0,T]\times\scr P $ for which $(\frac{\partial}{\partial t}+\mathbf{L}_t,\tilde{\scr F})$ is dissipative (hence closable) on the corresponding $L^1$ space and $\mathbf{L}_t$ is given by a time dependent vector field in the tangent bundle $(L^2(\R^d\rightarrow\R^d,\mu))_{\mu\in\scr P}$ over $\scr P $, determined by $b$ and $\sigma$, where $b,\sigma$ are arbitrary as in \eqref{eq:2.0}.
\end{rem}
\section{Correspondences and their consequences}
Consider the situation of Section 2.
\subsection{Correspondence of \eqref{eq:2.2'}  and \eqref{44EE} and weak well-posedness for McKean-Vlasov SDEs}
\begin{thm}\label{prp41}
Let $(\scr P_0,\scr A)$, $(\tilde{\scr P}_0,\tilde{\scr A})$ be two pairs having property {\bf (P)}. Then \eqref{eq:2.2'}  is well-posed in $\left((\scr P_0,\scr A),(\tilde{\scr P}_0,\tilde{\scr A})\right)$ (in the sense of Definition \ref{44DF1*}(2)) if and only if \eqref{44EE} is well-posed in $\left((\scr P_0,\scr A),(\tilde{\scr P}_0,\tilde{\scr A})\right)$ (in the sense of Definition \ref{44DF1}(ii)). In this case $\nu_{s,t}^{\zeta,\theta}=\scr L_{X_{s,t}^{\zeta,\theta}}$ for all $(s,\theta,\zeta)\in[0,\infty)\times\tilde{\scr P_0}\times\scr P_0$, $t\in[s,\infty)$.
\end{thm}
\begin{proof}
This is an immediate consequence of Remark \ref{rem24} and Lemma 2.12 in \cite{RWTR}. The latter lemma is indeed applicable by property {\bf (P)} above.
\end{proof}
\begin{cor}\label{thm42}
Let $\sigma=\bar\sigma$, $b=\bar b$ and let $(\scr P_0,\scr A)$ be a pair with property {\bf (P)}. If \eqref{eq:2.2'}  is well-posed in $\left((\scr P_0,\scr A),({\scr P}_0,{\scr A})\right)$ then \eqref{eq:1.8} is weakly well-posed in $(\scr P_0,\scr A)$.
\end{cor}
\begin{proof}
This is a direct consequence of Theorem \ref{prp41}.
\end{proof}
\subsection{Correspondence of \eqref{eq:2.2'}  and its linearization \eqref{eq:2.2''}}
\begin{thm}\label{thm43}
\begin{enumerate}
	\item[(i)] Let $s\geq0$ and $(\mu_t)_{t\geq s}$, $(\nu_t)_{t\geq s}\in C([s,\infty)\rightarrow\scr P)$. Then $(\mu_t,\nu_t)_{t\geq s}$ solves \eqref{eq:2.2'}  if and only if $\Lambda_t:=\nu_t\times\delta_{\mu_t}$, $t\geq s$, solves \eqref{eq:2.2''}.
	\item[(ii)] Let $s\geq0$ and let $\scr P_0,\tilde{\scr P}_0\in\scr B(\scr P)$ such that for each $\zeta\in\scr P_0$ there exists a solution $(\mu_{s,t}^\zeta)_{t\geq s}$ to \eqref{eq:1.3} with $\mu_{s,s}^\zeta=\zeta$ and such that for each $\theta\in\tilde{\scr P_0}$ there exists a solution $(\nu_{s,t}^{\zeta,\theta})_{t\geq s}$ to \eqref{44A2} with $\mu_{s,r}^\zeta$ replacing $\mu_r$, $r\geq s$, such that $\nu_{s,s}^{\zeta,\theta}=\theta$. Suppose that for every $t\in[s,\infty)$
	$$\scr P_0\ni\zeta\mapsto\mu_{s,t}^\zeta\in\scr P,\ \ \scr P_0\times\tilde{\scr P}_0\ni(\zeta,\theta)\mapsto\nu_{s,t}^{\zeta,\theta}\in\scr P$$ are Borel measurable. Then for every $\Lambda\in\scr P(\scr P_0\times\tilde{\scr P}_0)$(= all probability measures on $\scr P_0\times\tilde{\scr P}_0)$
$$ \Lambda_{s,t}^{\Lambda}:=\int_{\tilde{\scr P}_0\times\scr P_0}(\nu_{s,t}^{\zeta,\theta}\times\delta_{\mu_{s,t}^\zeta})\Lambda(\d \theta,\d\zeta),\ t\in[0,\infty),$$ is a solution to \eqref{eq:2.2''}.
\end{enumerate}
\end{thm}
\begin{proof}
(i): The proof is entirely analogous to the proof of Proposition 3.1.\\
(ii): The proof follows by (i) since \eqref{eq:2.2''}  is a linear Fokker-Planck equation (cf. the last paragraph in Section 3).
\end{proof}
\begin{rem}
 In fact one can prove  that if \eqref{eq:2.2'}  is well-posed then so is \eqref{eq:2.2''}. But for the proof one needs that any solution to \eqref{eq:2.2''}  is of the type as $\Lambda_t$, $t\geq s$, is in assertion (i) of Theorem \ref{thm43} above.
 The latter was, however, recently proved in \cite{MReh}.
\end{rem}

\subsection{Markov property of weak solutions to McKean-Vlasov equations}
In this section we fix two pairs $(\scr P_0,\scr A)$ and $(\tilde{\scr P}_0,\tilde{\scr A})$ with property {\bf (P)} and assume that
\beq\label{EQ431*}
\text{\eqref{eq:2.2'}  is well-posed in $\Big((\scr P_0,\scr A)$, $(\tilde{\scr P}_0,\tilde{\scr A})\Big)$.}
\end{equation}
For $(s,\theta,\zeta)\in[0,\infty)\times\tilde{\scr P}_0\times\scr P_0$ we define the laws
\begin{equation}\label{EQ432*}
\begin{aligned}
\mathbb{P}_{\zeta,(r,\theta)}&:=\mathbb{P}\circ(X_{r,\cdot}^{\mu_{s,r}^\zeta,\theta})^{-1},\ r\in[s,\infty), \\
\mathbb{P}_{\zeta,(r,x)}&:=\mathbb{P}_{\zeta,(r,\delta_x)}, \quad x\in\R^d, \text{ provided } \delta_x\in\tilde{\scr P}_0,
\end{aligned}
\end{equation}
on $C([r,\infty)\rightarrow\R^d)$, equipped with the $\sigma$-algebra $\scr G$ generated by all maps $\pi_t$, $t\geq r$, where $\pi_t$ is the evaluation map at t. In addition, for $t\in[s,\infty)$ we define
\begin{align}\label{EQ433*}
\scr G_{s,t}:=\sigma(\pi_u:\; u\in[s,t]).
\end{align}
Furthermore, we denote the corresponding expectations by $\mathbb{E}_{\zeta,( r,\theta)}$.\\ \\
As mentioned before, in general it is not possible to prove uniqueness of linear (or more so, nonlinear) Fokker-Planck equations for all initial probability measures $\theta$, in particular not for Dirac measures. Therefore, we only assume a restricted well-posedness in \eqref{EQ431*}. Hence (see \cite[Lemma 2.12]{RWTR}) also the martingale problems corresponding to linear Fokker-Planck equations are only restricted well-posed. Therefore, the standard fact, that well-posedness (i.e.~for all Dirac, hence all probability measures) implies that the corresponding family of probability measures $\mathbb{P}_x$ (= solution with initial marginal $\delta_x$), $x\in\R^d$, form a Markov process, is not applicable. Nevertheless, we shall prove that under condition \eqref{EQ431*} we have the Markov property for our laws $\mathbb{P}_{\zeta,(s,\theta)}$ defined above and as a consequence also for the laws of the solutions to the McKean-Vlasov equation \eqref{eq:1.8}. And this holds just assuming the integrability conditions \eqref{44A1**} and \eqref{44A1} on our coefficients.\\
As preparation for fixed $(s,\theta,\zeta)\in[0,\infty)\times\tilde{\scr P}_0\times\scr P_0$ and $r\in[ s,\infty)$ we disintegrate the measure $\mathbb{P}_{\zeta,(r,\nu_{s,r}^{\zeta,\theta})}$ with $\nu_{s,r}^{\zeta,\theta}$ as in Definition \ref{44DF1*}(2)(b) with respect to the map $\pi_r\colon C([r,\infty)\rightarrow\R^d)\rightarrow \R^d$ as follows
\begin{align}\label{EQ434*}
\mathbb{P}_{\zeta,(r,\nu_{s,r}^{\zeta,\theta})}(\d w)=p(x,\d w)\; \nu_{s,r}^{\zeta,\theta}(\d x),
\end{align}
where $p$ is a probability kernel from $\R^d$ to $C([r,\infty)\rightarrow\R^d)$ such that $p(x,\{\pi_r=x \})=1$ for all $x\in\R^d$. The existence of such a kernel follows by standard results on disintegration of measures.
\begin{thm}\label{thm45}
Assume that \eqref{EQ431*} holds, let $p$ be as in \eqref{EQ434*} and let $(s,\theta,\zeta)\in[0,\infty)\times\tilde{\scr P}_0\times\scr P_0$ and $r\in[s,\infty)$. Then for every $g\in \scr B_b(\R^d)$ and $t\in[r,\infty)$
\begin{align}\label{EQ435*}
\mathbb{E}_{\zeta,(s,\theta)}[g(\pi_t)|\; \scr G_{s,r}]=\int g(\pi_t(w))p(\pi_r,\d w)\quad \mathbb{P}_{\zeta,(s,\theta)}\text{-a.e.}
\end{align}
\end{thm}
\begin{proof}
Since for all $\varphi \in C_0^\infty(\R^d)$
\begin{align*}
\varphi(\pi_t)-\varphi(\pi_r)-\int_r^t L_{u,\mu_{s,u}^\zeta}\varphi \;\d u,\ t\geq r,
\end{align*}
is a $(\scr G_{r,t})_{t\geq r}$-martingale under $\mathbb{P}_{\zeta,(r,\nu_{s,r}^{\zeta,\theta})}$, it is elementary to check that for $\nu_{s,r}^{\zeta,\theta}$-a.e. $x\in\R^d$ this is also true under $p(x,\d w)$ defined in \eqref{EQ434*}. Hence this is also true for the measure
\begin{align}\label{EQ45Prime}
\mathbb{P}_\rho(\d w):=\int p(x,\d w)\rho(x)\; \d\nu_{s,r}^{\zeta,\theta}(\d x)
\end{align}
for every probability density $\rho$ with respect to $\nu_{s,r}^{\zeta,\theta}$, i.e.~$\mathbb{P}_\rho$ satisfies the martingale problem on $C([r,\infty)\rightarrow\R^d)$ for $L_{u,\mu_{s,u}^\zeta}$ with initial measure $\rho\nu_{s,r}^{\zeta,\theta}$. The latter follows from the fact that $p(x,\d w)$ is supported by $\{\pi_r=x\}$.\\
Now let $n\in\N$, $s\leq u_1\cdots\leq u_n\leq r$, $h\in\scr B_b((\R^d)^n)$, $h\geq 0$ and $g\in\scr B_b(\R^d)$. Define the factorized conditional expectation
\begin{align*}
\rho(x):=c\mathbb{E}_{\zeta,(s,\theta)}[h(\pi_{u_1},\dots,\pi_{u_n})|\; \pi_r]_{| \pi_r=x},\ x\in\R^d,
\end{align*}
where $c\in(0,\infty)$ so that its integral w.r.t.~$\nu_{s,r}^{\zeta,\theta}$ is equal to $1$. $\rho$ is uniquely defined $\nu_{s,r}^{\zeta,\theta}$-a.e. Now consider $\P_\rho$ defined in \eqref{EQ45Prime} for this $\rho$.
Furthermore, let
\begin{align*}
\bar{\rho}:=c\; h(\pi_{u_1},\dots,\pi_{u_n}).
\end{align*}
Then $\bar{\rho}$ is a probability density w.r.t.~$\mathbb{P}_{\zeta,(s,\theta)}$. We denote the image measure of $\bar{\rho}\cdot\mathbb{P}_{\zeta,(s,\theta)}$ under the natural projection of $C([s,\infty)\rightarrow\R^d)$ onto $C((r,\infty)\rightarrow \R^d)$ by $\mathbb{P_{\bar{\rho}}}$.\\ \\
\underline{Claim:} $ \P_{\bar{\rho}}=\P_\rho$.\\
It is easy to check that also $\P_{\bar{\rho}}$ satisfies the above martingale problem. Furthermore, also under $\P_{\bar{\rho}}$ the law of $\pi_r$ is equal to $\rho\cdot\nu_{s,r}^{\zeta,\theta}$. Therefore, by \eqref{EQ431*} and \cite[Lemma 2.12]{RWTR} the claim follows.\\ \\
By the Claim we have for $g\in\scr B_b(\R^d)$
\begin{align*}
&\E_{\zeta,(s,\theta)}[h(\pi_{u_1},\dots,\pi_{u_n})g(\pi_t)]\\
&=\frac1c\int g(\pi_t)\; \d\P_{\bar{\rho}}\\
&=\frac1c\int g(\pi_t)\; \d\P_\rho=\frac1c\;\E_{\zeta,(s,\theta)}\Big[\int_{\R^d}g(\pi_t(w))\; p(\pi_r,\d w)\; \rho(\pi_r)\Big]\\
&=\E_{\zeta,(s,\theta)}\Big[\int_{\R^d}g(\pi_t(w))\; p(\pi_r,\d w)\; h(\pi_{u_1},\dots,\pi_{u_n})\Big].
\end{align*}
Now \eqref{EQ435*} follows by a monotone class argument.
\end{proof}
\begin{cor}\label{cor46}
Let $\sigma=\bar\sigma$, $b=\bar b$ and assume that \eqref{EQ431*} holds with $\scr P_0=\tilde{\scr P}_0$. Then for every $(s,\zeta)\in[0,\infty)\times\scr P_0$ the law $\P_{\zeta,(s,\zeta)}$ of the solution of the McKean-Vlasov equation \eqref{eq:1.8} started from $s$ at $\zeta$ is Markov, i.e.~satisfies \eqref{EQ435*} with $p$ as in \eqref{EQ434*}.
\end{cor}
\begin{rem}\label{rem47}
We note that in the situation of Corollary \ref{cor46} we obviously have that $\nu_{s,t}^{\zeta,\zeta}=\mu_{s,t}^\zeta$ for all $0\leq s\leq t$ and $\zeta\in\scr P_0$.
\end{rem}
\subsection{The Markov process on $\R^d\times\scr P $}
In this and the next subsection we fix a pair $(\scr P_0,\scr A)$ with property {\bf (P)} (see Section 2) and set
\begin{align}\label{EQ431}
\tilde{\scr P}_0:=\scr P,\quad \tilde{\scr A}:=C([0,\infty)\rightarrow\scr P),
\end{align}
so obviously $(\tilde{\scr P_0},\tilde{\scr A})$ is also a pair with property {\bf (P)}. We also assume throughout this and the next subsection
\beq\label{EQ432}
 \eqref{eq:2.2'}{\rm\ is\ well-posed\ in\ }((\scr P_0,\scr A),(\tilde{\scr P}_0,\tilde{\scr A})) {\rm\ with\ }(\tilde{\scr P}_0,\tilde{\scr A})  {\rm\ defined\ as\ in\ }\eqref{EQ431}.
\end{equation}
For $(s,x,\zeta)\in[0,\infty)\times\R^d\times\scr P_0$, as in Definition \ref{44DF1*}(2)(a), (b), we denote the corresponding solutions by $(\mu_{s,t}^\zeta,\nu_{s,t}^{\zeta,\delta_x})_{t\geq s}$.
For $x\in\R^d$, $\zeta\in\scr P_0$, $s\in[0,\infty)$ define for $t\in[s,\infty)$
\begin{align}\label{EQ433}
\mathbf{P}_{s,t}(x,\zeta;\d y\d\mu)=(\nu_{s,t}^{\zeta,\delta_x}\times\delta_{\mu_{s,t}^\zeta})(\d y\d\mu),
\end{align}
i.e. for $G:\R^d\times\scr P_0\rightarrow[0,\infty)$ Borel-measurable
\begin{align*}
\int_{\R^d}\int_{\scr P_0} G(y,\mu)\mathbf{P}_{s,t}(x,\zeta;\d y\d\mu)=\int_{\R^d}G(y,\mu_{s,t}^\zeta)\nu_{s,t}^{\zeta,\delta_x}(\d y).
\end{align*}
Then by Theorem \ref{thm43}(i) it follows that $\mathbf{P}_{s,t}(x,\zeta;\d y\d\mu)$ solves \eqref{eq:2.2''}  starting from time $s$ and furthermore
\begin{align}\label{EQ434}
\mathbf{P}_{s,s}(x,\zeta;\d y\d\mu)=(\delta_x\times\delta_\zeta)(\d y\d\mu).
\end{align}
\begin{prp}\label{P4.8}
Suppose \eqref{EQ432} holds. Then the family of probability measures $\mathbf{P}_{s,t}(x,\zeta;\cdot)$, $0\leq s\leq t$, $(x,\zeta)\in\R^d\times\scr P_0$, is a Markov transition kernel on $\R^d\times\scr P_0$, i.e., it satisfies the following properties:
\begin{enumerate}
\item[$(C_1)$] $\PP_{s,t}(\cdot;A) $ is Borel-measurable in $(x,\zeta)$ for any $0\le s\le t$, $A\in\scr P(\R^d\times\scr P_0)$ and $\PP_{s,s}(x,\zeta;\cdot)= \dd_{(x,\zeta)}$, the Dirac measure at $(x,\zeta)$, for  any $s\ge 0$ and $(x,\zeta)\in \R^d\times\scr P_0$.
 \item[$(C_2)$] The Chapman-Kolmogorov equations hold, i.e. for all $0\le s<r<t$ and $(x,\zeta)\in \R^d\times\scr P_0$
\begin{align}\label{EQ435}
 \PP_{s,t}(x,\zeta;\cdot)= \int_{\R^d\times\scr P_0} \PP_{r,t}(y,\mu;\cdot) \PP_{s,r}(x,\zeta; \d y, \d \mu).
\end{align}
\end{enumerate}
\end{prp}
\begin{proof}
$(C_1)$ is obvious, so it only remains to prove the Chapman-Kolmogorov equations.\\
For $0\leq s\leq r \leq t$ and $(x,\zeta)\in\R^d\times\scr P_0$ the right hand side of \eqref{EQ435} is equal to
\begin{align*}
&\int_{\R^d}\mathbf{P}_{r,t}(y,\mu_{s,r}^\zeta;\cdot)\; \nu_{s,r}^{\zeta,\delta_x}(\ \d y)\\
&=\int_{\R^d}\left(\nu_{r,t}^{\mu_{s,r}^\zeta,\delta_y}\times\delta_{\mu_{r,t}^{\mu_{s,r}^\zeta}}\right)\ \nu_{s,r}^{\zeta,\delta_x}(\ \d y)\\
&=\left(\int_{\R^d}\nu_{r,t}^{\mu_{s,r}^\zeta,\delta_y}\; \nu_{s,r}^{\zeta,\delta_x}(\ \d y)\right)\times\delta_{\mu_{s,t}^\zeta}\\
&=\nu_{s,t}^{\zeta,\delta_x}\times\delta_{\mu_{s,t}^\zeta}=\mathbf{P}_{s,t}(x,\zeta;\cdot)
\end{align*}
where we used the well-posedness of \eqref{eq:2.2'}, more precisely Definition \ref{44DF1*}(2)(a), (b) in the second and third equality respectively.
\end{proof}
A Markov transition kernel on $\R^d\times\scr P_0$ determines a family of Markov operators $\{\PP_{s,t}: 0\le s\le t\}$ on      $\B_b(\R^d\times\scr P_0)$, the Banach space of bounded Borel-measurable functions on $\R^d\times \scr P_0$:
\begin{align}\label{EQ46}
\PP_{s,t} f(x,\zeta):= \int_{\R^d\times\scr P_0} f(\xi)\, \PP_{s,t}(x,\zeta;\d\xi),\ \  (x,\zeta)\in \R^d\times \scr P_0, f\in  \B_b(\R^d\times\scr P_0).
\end{align}
Conditions $(C_1)$ and $(C_2)$ are equivalent to $\PP_{s,s}f=f$ and the semigroup property
$$\PP_{s,t}= \PP_{s,r}\PP_{r,t}, \ \ 0\le s\le r\le t.$$
By \eqref{EQ432} and Theorem \ref{prp41} we have that \eqref{44EE} is well-posed in $((\scr P_0,\scr A), (\tilde{\scr P_0},\tilde{\scr A}))$, with $(\tilde{\scr P_0},\tilde{\scr A})$ as in \eqref{EQ431}. For $(s,x,\zeta)\in[0,\infty)\times\R^d\times\scr P_0$, as in Definition \ref{44DF1}(ii) we denote the corresponding solutions by $(X_{s,t}^{\zeta,\delta_x},\mu_{s,t}^\zeta)_{t\geq s}$ on $(\Omega,\scr F,(\scr F_t)_{t\geq0},\mathbb{P})$ with $(\scr F_t)$-Brownian motion $(W_t)_{t\geq0}$ (see Section 2). We note that the stochastic basis and the Brownian motion depend on $(s,x,\zeta)$, but for simplicity we do not express this in the notation.

Our next aim is to prove that the laws of $(X_{s,t}^{\zeta,\delta_x},\mu_{s,t}^{\zeta})_{t\geq s}$, $(s,x,\zeta)\in[0,\infty)\times\R^d\times\scr P_0$, form a Markov process with Markov transition kernel, $\mathbf{P}_{s,t}$, $0\leq s\leq t$, defined in \eqref{EQ433}.
For $(s,\theta,\zeta)\in[0,\infty)\times\scr P\times\scr P_0$ we define the laws
\begin{equation}
\begin{aligned}\label{EQ436}
\mathbb{P}_{(s,\theta,\zeta)}&:=\mathbb{P}\circ(X_{s,\cdot}^{\zeta,\theta},\mu_{s,\cdot}^{\zeta})^{-1}=\left(\P\circ(X_{s,\cdot}^{\zeta,\theta})^{-1}\right)\times\delta_{\mu_{s,\cdot}^{\zeta}}=\P_{\zeta,(s,\theta)}\times\delta_{\mu_{s,\cdot}^\zeta}\ \text{ and}\\
\mathbb{P}_{(s,x,\zeta)}&:=\mathbb{P}_{(s,\delta_x,\zeta)},\ x\in \R^d,
\end{aligned}
\end{equation}
on $C([s,\infty)\rightarrow\R^d\times\scr P_0)=C([s,\infty)\rightarrow\R^d)\times C([s,\infty)\rightarrow \scr P_0).$ We denote the canonical projection on this path space by $\pi_t^0$, $t\in[s,\infty)$, and equip it with the $\sigma$-algebra $\scr G^0$ generated by these projections. In addition, we define for $t\in[s,\infty)$
\begin{align*}
\scr G_{s,t}^0:=\sigma(\pi_u^0:\ u\in[s,t]).
\end{align*}
Furthermore, we denote the corresponding expectations by $\mathbb{E}_{(s,\theta,\zeta)}$.

\begin{thm}\label{thm46}
Suppose that \eqref{EQ432} holds. Let $({\bf P}_{s,t})_{t\ge s\ge 0}$ be in $\eqref{EQ46}.$ Then for any $(s,x,\zeta)\in[0,\infty)\times\R^d\times\scr P_0$, $t\in[s,\infty)$, and $r\in[s,t]$:
\begin{enumerate}
\item[$($i$)$] $\mathbb{P}_{(s,x,\zeta)}\circ\pi_t^{-1}=\mathbf{P}_{s,t}(x,\zeta;\cdot).$
\item[$($ii$)$] $\mathbb{P}_{(s,x,\zeta)}$-a.s. we have for every $A\in\scr B(\R^d\times\scr P_0)$
\begin{align}\label{EQ438}
\mathbb{P}_{(s,x,\zeta)}[\pi_t^0\in A| \scr G_r^0]=\mathbb{P}_{(r,\pi_r^0)}[\pi_t^0\in A]=\mathbf{P}_{r,t}(\pi_r^0;A),
\end{align}
i.e., $\mathbb{P}_{(s,x,\zeta)}$, $(s,x,\zeta)\in[0,\infty)\times\R^d\times\scr P_0$, form a Markov process with Markov transition kernel $\mathbf{P}_{s,t}$, $0\leq s\leq t$.
\end{enumerate}
\end{thm}
\begin{proof}
(i) is obvious from the definitions.\\
(ii) Since by our assumptions we have weak uniqueness for the second equation in \eqref{44EE}, by \cite[Theorem 6.2.2]{RWSV} and Theorem \ref{prp41} above we know that for every $s\leq r\leq t$, $g\in \scr B_b(\R^d)$
\begin{align}\label{EQ439}
\mathbb{E}_{\zeta,(s,x)}[g(\pi_t)|\scr G_{s,r}]=\mathbb{E}_{\zeta,(r,\pi_r)}[g(\pi_t)]=\int_{\R^d}g\ \d\nu_{r,t}^{\mu_{s,r}^{\zeta},\delta_{\pi_r}}\ \  \mathbb{P}_{\zeta,(s,x)}\text{-a.s.}
\end{align}
For $n\in\N$; $s\leq u_1\leq\cdots\leq u_n\leq r$, and $H\in\scr B((\R^d\times\scr P_0)^n)$, we have that $\mathbb{P}_{(s,x,\zeta)}$-a.s. with $G:=1_A$
\begin{align*}
&\mathbb{E}_{(s,x,\zeta)}[H(\pi_{u_1}^0,\dots,\pi_{u_n}^0)G(\pi_t^0)]\\
=&\mathbb{E}[H((X_{s,u_1}^{\zeta,\delta_x},\mu_{s,u_1}^\zeta),\dots,(X_{s,u_n}^{\zeta,\delta_x},\mu_{s,u_n}^\zeta))G(X_{s,t}^{\zeta,\delta_x},\mu_{s,t}^{\zeta})]\\
=&\mathbb{E}_{\zeta,(s,x)}[H((\pi_{u_1},\mu_{s,u_1}^\zeta),\dots,(\pi_{u_n},\mu_{s,u_n}^\zeta))G(\pi_t,\mu_{s,t}^\zeta)]\\
\underset{\eqref{EQ439}}{=}&\mathbb{E}[H((X_{s,u_1}^{\zeta,\delta_x},\mu_{s,u_1}^{\zeta}),\dots,(X_{s,u_n}^{\zeta,\delta_x},\mu_{s,u_n}))\mathbb{E}_{\zeta,(r,X_{s,r}^{\zeta,\delta_x})}[G(\pi_t,\mu_{s,t}^\zeta)]].
\end{align*}
But since by the well-posedness of the first equation in \eqref{44EE} in $(\scr P_0,\scr A)$ we have the flow property
\begin{align*}
\mu_{s,t}^\zeta=\mu_{r,t}^{\mu_{s,r}^\zeta},
\end{align*}
\eqref{EQ432*} and \eqref{EQ436} imply that for $\mathbb{P}$-a.e. $\omega\in\Omega$
\begin{align*}
&\mathbb{E}_{\zeta,(r,X_{s,r}^{\zeta,\delta_x}(\omega))}[G(\pi_t,\mu_{s,t}^\zeta)]\\
=&\mathbb{E}[G(X_{r,t}^{\mu_{s,r}^{\zeta},\delta_{X_{s,r}^{\zeta,\delta_x}(\omega)}},\mu_{r,t}^{\mu_{s,r}^\zeta})]\\
=&\mathbb{E}_{(r,X_{s,r}^{\zeta,\delta_x}(\omega),\mu_{s,r}^\zeta)}[G(\pi_t^0)].
\end{align*}
Hence altogether
\begin{align*}
&\mathbb{E}_{(s,x,\zeta)}[H(\pi_{u_1}^0,\dots,\pi_{u_n}^0)G(\pi_t^0)]\\
=&\mathbb{E}_{(s,x,\zeta)}[G(\pi_{u_1}^0,\dots,\pi_{u_n}^0)\mathbb{E}_{(r,\pi_r^0)}[G(\pi_t^0)]].
\end{align*}
Hence the first inequality in \eqref{EQ438} follows by a monotone class argument and the second is obvious by (i).
\end{proof}
We note that since \eqref{EQ432} is a stronger assumption than \eqref{EQ431*} we get a more explicit way to formulate the Markov property for the laws of the (unique) weak solution to the McKean-Vlasov equation \eqref{eq:1.8} (see \eqref{EQ4318}, \eqref{EQ4319} below).
\begin{align}\label{EQ4310}
b=\bar{b},\ \sigma=\bar{\sigma}.
\end{align}
Then for $(s,\zeta)\in[0,\infty)\times\scr P_0$ by \eqref{EQ432} we have that (using our notations above) for every $\theta\in\scr P$ there exists a unique in law weak solution $(X_{s,t}^{\zeta,\theta})_{t\geq s}$ of the second equation in \eqref{44EE} with $b,\sigma$ replacing $\bar{b}$ and $\bar{\sigma}$ respectively on some probability  space $(\Omega,\scr F,\mathbb{P})$ with normal filtration $(\scr F_t)_{t\geq 0}$ and $(\scr F_t)$-Brownian motion $(W_t)_{t\geq 0}$ (see Section 2). Again all these three quantities depend on $(s,\theta,\zeta)$, but for simplicity we do not express this in the notation. In particular, for $\theta:=\zeta$ we obtain a weak solution to \eqref{eq:1.8} with marginal law $\zeta$ at time $s$.
\begin{lem}\label{lem47}
For all $(s,\theta,\zeta)\in[0,\infty)\times\scr P\times \scr P_0$ we have
\begin{align*}
\mathbb{P}_{\zeta,(s,\theta)}=\int_{\R^d}\mathbb{P}_{\zeta,(s,x)}\theta(\d x).
\end{align*}
\end{lem}
\begin{proof}
The proof is standard, but we repeat the argument here: The two probability measures in the assertion solve the martingale problem with initial condition $\zeta$ for the Kolmogorov operator $L_{t,\zeta}$ defined in \eqref{eq:1.2} in the sense of \cite{RWSV}. But by \eqref{EQ432} and Theorem \ref{prp41} this martingale problem has a unique solution. So, both measures coincide.
\end{proof}
\begin{thm}\label{thm48}
Let \eqref{EQ432} hold and let $(s,\zeta)\in[0,\infty)\times\scr P_0$. Then:
\begin{enumerate}
	\item[$($i$)$] The unique weak solution to \eqref{eq:1.8} with marginal law $\zeta$ at $s$ has the Markov property, i.e. for every $s\leq r\leq t$, $g\in\scr B_b(\R^d)$,
	\begin{align}\label{EQ4318}
	\mathbb{E}_{\zeta,(s,\zeta)}[g(\pi_t)|\scr G_{s,r}]=\mathbb{E}_{\zeta,(r,\pi_r)}[g(\pi_t)]=\int_{\R^d}g\ \d\nu_{r,t}^{\mu_{s,r}^{\zeta},\delta_{\pi_r}}\ \text{ $\mathbb{P}_{\zeta,(s,\zeta)}$-a.s.}
	\end{align}
	\item[$($ii$)$] $\mathbb{P}_{(s,\zeta,\zeta)}$ is Markov. More precisely, for $s\leq r\leq t$ and $A\in\scr B(\R^d\times\scr P_0)$
	\begin{align}\label{EQ4319}
	\mathbb{P}_{(s,\zeta,\zeta)}[\pi_t^0\in A|\scr G_r^0]=\mathbf{P}_{r,t}(\pi_r^0;A)\ \text{ $\mathbb{P}_{(s,\zeta,\zeta)}$-a.s.}
	\end{align}
\end{enumerate}
\end{thm}
\begin{proof}
(i): The assertion immediately follows by \eqref{EQ439} and Lemma \ref{lem47}.\\
(ii): As an easy consequence of Lemma \ref{lem47} we get that
\begin{align*}
\mathbb{P}_{(s,\zeta,\zeta)}=\int \mathbb{P}_{(s,x,\zeta)}\zeta(\d x).
\end{align*}
Then the assertion immediately follows from Theorem \ref{thm46}(ii).
\end{proof}
\subsection{Ergodicity}
We recall that we still assume \eqref{EQ432} with $(\tilde{\scr P}_0,\tilde{\scr A})$ as in \eqref{EQ431}. In this subsection we assume additionally that for our coefficients from \eqref{eq:2.0} we have
\begin{align}\label{eq:4.15}
	b, \bar{b}, \sigma \text{ and } \bar{\sigma} \text{ do not depend on } t\in[0,\infty).
\end{align}
Then due to \eqref{EQ432} for our Markov transition kernel $\mathbf{P}_{s,t}$, $0 \leq s \leq t$, defined in \eqref{EQ433} we have that $ \mathbf{P}_{s,t} = \mathbf{P}_{0,t-s}$, i.e.~it is time-homogeneous.
Set
\begin{align}\label{eq:4.16}
	\mathbf{P}_{t}:= \mathbf{P}_{0,t}, \quad t\geq 0.
\end{align}
Then by $(C_2)$ this is a semigroup of probability kernels on $\mathbb R^d \times \scr P_0$ or equivalently (see \eqref{EQ46}) of operators on $\scr B_b (\mathbb R^d \times \scr P_0)$.
We recall that $\Lambda \in \scr P(\mathbb R^d \times \scr P_0)$ is called $(\mathbf{P}_{t})$-invariant if
\begin{align}\label{eq:4.17}
	\int_{\mathbb R^d \times \scr P_0} \mathbf{P}_{t} G \, \d \Lambda = \int_{\mathbb R^d \times \scr P_0} G \, \d \Lambda,
\end{align}
for all $G \in \scr B_b (\mathbb R^d \times \scr P_0)$, $t>0$, and that $(\mathbf{P}_{t})_{t>0}$ is called ergodic if there exists $\Lambda \in \scr P(\mathbb R^d \times \scr P_0)$ such that for all $\tilde \Lambda \in \scr P(\mathbb R^d \times \scr P_0)$ (or equivalently for all $\tilde{\Lambda}=\delta_x\times\delta_\zeta,\ (x,\zeta)\in\R^d\times\scr P_0)$
\begin{align}\label{eq:4.18}
	\lim\limits_{t \rightarrow \infty} \int_{\mathbb R^d \times \scr P_0} \mathbf{P}_{t} G \, \d \tilde\Lambda = \int_{\mathbb R^d \times \scr P_0} G \, \d \Lambda,
\end{align}
for all $G\in C_b(\mathbb R^d \times \scr P_0)$, where the latter denotes the set of all bounded continuous functions on $\mathbb R^d \times \scr P_0$.
In this case $\Lambda$ is then obviously the unique $(\mathbf{P}_{t})$-invariant measure.
From the definition of $\mathbf{P}_{t}$ we now obtain the following characterization.

\begin{thm}
Assume \eqref{EQ432} and \eqref{eq:4.15} hold.
Then the following assertions are equivalent:
\begin{itemize}
	\item[(i)] $(\mathbf{P}_{t})_{t>0}$ is ergodic.
	\item[(ii)] There exist $\nu_\infty \in \scr P$, $\mu_\infty \in \scr P_0$ such that for every $(x, \zeta) \in \mathbb R^d \times \scr P_0$
				\begin{align*}
					\mu_{0,t}^\zeta \rightarrow \mu_\infty \text{ and } \nu_{0,t}^{\zeta, \delta_x} \rightarrow \nu_\infty \text{ weakly as } t \rightarrow \infty.
				\end{align*}
\end{itemize}
In this case $\mu_\infty$ and $\nu_\infty$ are uniquely determined and the unique $(\mathbf{P}_{t})$-invariant measure $\Lambda$ is given by $\Lambda = \nu_\infty \times \delta_{\mu_\infty}$.
Furthermore, if $b = \bar b$ and $\sigma = \bar \sigma$, then $\mu_\infty = \nu_\infty$.
\end{thm}

\begin{proof}
$(ii) \Rightarrow (i)$: This is easy to see from the definition of $\mathbf{P}_{t}$ in \eqref{eq:4.16} and \eqref{EQ433}, and the unique $(\mathbf{P}_{t})$-invariant measure is $\Lambda:= \nu_\infty \times \delta_{\mu_\infty}$.\\
$(i) \Rightarrow (ii)$: Let $\Lambda$ be the $(\mathbf{P}_t)$-invariant measure such that \eqref{eq:4.18} holds.
Let $\Pi_1 \colon \mathbb R^d \times \scr P_0 \longrightarrow \mathbb R^d$, and $\Pi_2 \colon \mathbb R^d \times \scr P_0 \longrightarrow \scr P_0$ be the canonical projections and define
\begin{align*}
	\nu_\infty := \Lambda \circ \Pi_1^{-1} (\in \scr P) \text{ and } \Lambda_2 := \Lambda \circ \Pi_2^{-1} (\in \scr P(\scr P_0)).
\end{align*}
Then for every $(x,\zeta) \in  \mathbb R^d \times \scr P_0$ as $t \rightarrow \infty$
\beq\label{DDE0}
	\nu_{0,t}^{\zeta, \delta_x} \longrightarrow \nu_\infty \text{ weakly in } \scr P
\end{equation}
and
\beq\label{DDE}
	\delta_{\mu_{0,t}^\zeta} \longrightarrow \Lambda_2\ \text{weakly in } \scr P(\scr P_0) \ ({\rm hence\ also\ in\ }\scr P(\scr P)\ {\rm by\ extending\ }\LL_2\ {\rm by\ zero}).
\end{equation}
We claim that \eqref{DDE} implies     $\Lambda_2 = \delta_{\mu_\infty}$ for some  $\mu_\infty \in \scr P_0$. Combining this with   \eqref{eq:4.18} for $\tt\Lambda=\zeta\times \dd_x$ and using \eqref{DDE0}, we conclude that $\LL=\nu_\infty\times \dd_{\mu_\infty}$ as desired.

To prove the claim, let $\mu_\infty\in\scr P$ be defined by
$$\mu_\infty(A):= \int_{\scr P_0}\mu(A)\LL_2(\d\mu)=\int_{\scr P}\mu(A)\LL_2(\d\mu),\ \ A\in \scr B(\R^d),$$where $\LL_2$ is extended to $\scr P$ by zero, i.e. $\LL_2(\scr P\setminus\scr P_0)=0.$
For any $h\in C_b(\R^d)$ we take $F(\mu)=\mu(h):=\int h\d\mu, \mu\in \scr P$.
Then $F\in C_b(\scr P)$. Since  $\delta_{\mu_{0,t}^\zeta} \to \Lambda_2 \text{ weakly in } \scr P(\scr P)$, we have
$$\lim_{t\to\infty}\mu_{0,t}^{\zeta}(h)= \lim_{t\to\infty}\dd_{\mu_{0,t}^{\zeta}}(F)
=\LL_2(F)=\mu_\infty(h).$$ So, $\mu_{0,t}^{\zeta}\to\mu_\infty$ weakly in $\scr P$, and hence  $  \dd_{\mu_{0,t}^{\zeta}}\to\dd_{\mu_\infty}$weakly in $\scr P(\scr P)$ as $t\to\infty$. Combining this with \eqref{DDE} we prove $\LL_2=\dd_{\mu_\infty}$ and $\mu_\infty\in \scr P_0$ since $\LL_2$ is supported on $\scr P_0$.

We note that for every $(\zeta,\theta)\in\scr P_0\times\scr P$
\begin{align*}
\nu_{0,t}^{\zeta,\theta}=\int \nu_{0,t}^{\zeta,\delta_x}\ \theta(\d x),\ \ t\geq0,
\end{align*}
since the second equation in \eqref{eq:2.2'}  is linear and well-posed. Hence from (ii) we have
\begin{align*}
\nu_{0,t}^{\zeta,\theta}\longrightarrow\nu_\infty\quad \text{weakly as $t\rightarrow\infty$.}
\end{align*}
Now assume that $b=\bar{b}$ and $\sigma=\bar\sigma$. Then for all $\zeta\in\scr P_0$
\begin{align*}
\nu_\infty&=w-\lim_{t\rightarrow\infty}\nu_{0,t}^{\zeta,\zeta}\\
&=w-\lim_{t\rightarrow\infty} \mu_{0,t}^\zeta\\
&=\mu_\infty,
\end{align*}
where we used Remark \ref{rem22} in the second equality. This completes the proof.
\end{proof}
\section{Application to nonlinear distorted Brownian motion}
In this section we want to apply our results to the so-called nonlinear distorted Brownian motion (NLDBM) in which case the nonlinear Fokker-Planck equation \eqref{eq:1.3} is a porous media equation perturbed by a nonlinear transport term. We shall give details below, but want to stress already now that in this case the solutions of \eqref{eq:1.3} are absolutely continuous with respect to Lebesgue measure $\d x$, if so is the initial condition, i.e.~, $u(\d x)=u(t,x)\d x $, $t\geq0$. Furthermore, in the case of NLDBM the coefficients $b$, $\sigma$ in \eqref{eq:2.0} (to be introduced below explicitly) depend `` Nemytski type''  on $\mu_t$, more precisely for $(t,x,\mu)\in[0,\infty)\times\R^d\times\scr P_0$, where
\begin{align}\label{EQ51}
\scr P_0:=\Big\{\mu\in\scr P|\; u_\mu:=\frac{\d \mu}{\d x}\in L^\infty(\R^d,\d x) \Big\},
\end{align}
we have
\begin{equation}\label{EQ52}
\begin{aligned}
b(t,x,\mu)&=\tilde{b}\Big(t,x,\frac{\d\mu}{\d x}(x)\Big),\\
\sigma(t,x,\mu)&=\tilde{\sigma}\Big(t,x,\frac{\d\mu}{\d x}(x)\Big)
\end{aligned}
\end{equation}
for some Borel measurable functions
\begin{align*}
\tilde{b}\colon[0,\infty)\times\R^d\times\R\longrightarrow\R^d,
\end{align*}
and
\begin{align*}
\tilde{\sigma}\colon[0,\infty)\times\R^d\times\R\longrightarrow\R^d\otimes\R^m.
\end{align*}
Of course, for $x\in\R^d$ we have to choose the $\d x$-version of $\frac{\d \mu}{\d x}$ in such a way that the maps
\begin{align*}
[0,\infty)\times\R^d\times\scr P_0\ni(t,x,\mu)&\longmapsto\tilde{b}\Big(t,x,\frac{\d \mu}{\d x}(x)\Big),\\
[0,\infty)\times\R^d\times\scr P_0\ni(t,x,\mu)&\longmapsto\tilde{\sigma}\Big(t,x,\frac{\d \mu}{\d x}(x)\Big)
\end{align*}
are Borel-measurable. But this is easily achieved by looking at the $\d x$-version obtained by defining $\frac{\d\mu}{\d x}\equiv0$ on the complement of its Lebesgue points. Below we shall always take this version without further mentioning it.\\

To introduce $\tilde{b}$ and $\tilde{\sigma}$ concretely in the case of NLDBM we consider maps $\beta\colon\R\rightarrow\R$, $D\colon\R^d\rightarrow\R^d$ and $b\colon\R\rightarrow\R$ satisfying the following hypotheses:
\begin{enumerate}[label=(H)]
\item \label{condH}
\begin{enumerate}[label=(\roman*)]
\item $\beta\in C^1(\R)$, $\beta(0)=0$, $\gamma\leq\beta'(r)\leq\gamma_1$, $\forall r\in\R$, for $0<\gamma<\gamma_1<\infty$.
\item $b\in C_b(\R)\cap C^1(\R)$.
\item $D\in C_b(\R^d;\R^d)\cap W^{1,\infty}(\R^d;\R^d).$
\item $D= -\nabla\Phi$, where $\Phi\in C^1(\R^d)$, $\Phi\geq1$, $\underset{|x|_d\rightarrow\infty}{\lim}\Phi(x)=+\infty$ and there exists $m\in[2,\infty)$ such that $\Phi^{-m}\in L^1(\R^d)$.
\end{enumerate}
\end{enumerate}
A typical example for $\Phi$ is
\begin{align}\label{EQ53}
\Phi(x)=C(1+|x|^2)^\alpha,\ x\in\R^d,
\end{align}
with $\alpha\in(0,\frac12]$.\\
Then we have for the corresponding ($t$-independent) Kolmogorov operator \eqref{eq:1.3} for $\mu\in\scr P_0$:
\begin{align}\label{EQ54}
L_\mu h(x):=\frac12\frac{\beta(u_\mu(x))}{u_\mu(x)}\Delta h(x)+b(u_\mu(x))D(x)\cdot\nabla h(x),\ x\in\R^d,\ h\in C_0^2(\R^d),
\end{align}
where we set $\frac{\beta(0)}{0}:=\beta'(0)$, i.e.~compared to \eqref{eq:1.3} we have for $(t,x,\mu)\in[0,\infty)\times\R^d\times\scr P_0$
\begin{equation}\label{EQ55}
\begin{aligned}
\sigma(t,x,\mu)&=\frac{\beta(u_\mu(x))}{u_\mu(x)}\mathrm{Id},\\
b(t,x,\mu)&=b(u_\mu(x))D(x).
\end{aligned}
\end{equation}
Here $\mathrm{Id}$ denotes the identity matrix on $\R^d$. Hence \eqref{eq:1.2} becomes an equation for density $\mu(t,\cdot)=\frac{\d\mu_t}{\d x}$ and then reads as
\begin{align}\label{EQ56}
\partial_t u(t,x)=\frac12\Delta\beta(u(t,x))-\mathrm{div}(D(x)b(u(t,x))u(t,x)).
\end{align}
In this section we consider the case $b=\bar{b}$, $\sigma=\bar{\sigma}$. So, the first equation in \eqref{eq:2.2'}  is just \eqref{EQ56} and the second reads
\begin{align}\label{EQ57}
\partial_t\nu_t=\frac12\Delta\Big(\frac{\beta(u(t,\cdot))}{u(t,\cdot)}\nu_t\Big)-\mathrm{div}(D(\cdot)b(u(t,\cdot))\nu_t),
\end{align}
(as always in this paper) meant in the weak sense.

In particular, if $\zeta:=u_0\d x\in\scr P_0$ is the initial condition for the solution of \eqref{EQ56}, then according to our notation in Definition \ref{44DF1*}(2)(b) the solution to \eqref{EQ57} starting from $s\in[0,\infty)$ with $\theta\in\scr P_0$ is denoted by $\nu_{s,t}^{\zeta,\theta}$, $t\geq s$.\\
The corresponding McKean-Vlasov equation reads:
\begin{align}\label{EQ58}
\d X_t=b(\scr L_{X_t}(X_t))D(X_t)\d t + \frac12\frac{\beta(\scr L_{X_t}(X_t))}{\scr L_{X_t}(X_t)}\d W_t.
\end{align}
We recall that by Theorem \ref{thm24} we have that \eqref{EQ56} and \eqref{EQ58} are equivalent (with $\mu_t=\scr L_{X_t},\ t\geq0$). \eqref{EQ56} and consequently weak solutions of \eqref{EQ58} have been analyzed in \cite{RWBR19}, \cite{RWVB19} and the solution process to \eqref{EQ58} has been named \textit{nonlinear distorted Brownian motion}, because in the linear case, i.e.~$b=$constant and $\beta=$identity, \eqref{EQ58} reduces to the SDE for distorted Brownian and \eqref{EQ54} to the corresponding linear Kolmogorov operator.

Now let
\begin{align*}
\scr A:= C([0,\infty)\rightarrow\scr P_0)\cap L^\infty([0,\infty)\times\R^d).
\end{align*}
Then obviously condition {\bf (P)} in Section 2 holds. The following is our main result on NLDBM.
\begin{thm}\label{thm51}
Assume {\rm\ref{condH}(i)-(iv)} hold. Then:
\begin{enumerate}
\item[$($i$)$] \eqref{EQ58} is weakly well-posed in $(\scr P_0,\scr A)$.
\item[$($ii$)$]  Nonlinear distorted Brownian motion has the Markov property. More precisely, for every $(s,\zeta)\in[0,\infty)\times\scr P_0$ the law $\P_{\zeta,(s,\zeta)}$ of the (unique weak) solution of \eqref{EQ58} started from $s$ at $\zeta$ satisfies \eqref{EQ435*} with $p$ as in \eqref{EQ434*}.
\end{enumerate}
\end{thm}
\begin{proof}
(i): By \cite[Proposition 2.2]{RWBR19} and \cite[Theorems 2.1 and 3.1]{RWVB19} in the case of NLDBM we have well-posedness for \eqref{eq:2.2'}  in $(\scr P_0,\scr A)$ (as defined above). Hence the assertion follows by Corollary \ref{thm42} above.\\
(ii) is then a consequence of (i) and Corollary \ref{cor46}.
\end{proof}
\begin{rem}
If additionally we make the following assumptions:
\begin{enumerate}[label=$($H$)$]
\item \label{condH2}
\begin{enumerate}
	\item[$($v$)$] $b(r)\geq b_0>0$ for all $r\in\R$
	\item[$($vi$)$] $\gamma_1\Delta\Phi-b_0|\nabla\Phi|^2\leq0$,
\end{enumerate}
\end{enumerate}
then it has been proved in \cite{RWBR19} (see Theorems 6.1 and 6.5 in there) that \eqref{EQ56} has a stationary solution which is unique in a slightly restricted class of probability densities in $L^1(\R^d,\d x)$. Hence nonlinear distorted Brownian motion as a unique invariant measure in this class.
\end{rem}
\section{Exponential ergodicity of ${\bf P}_t$}
In this section, we let $b(t,x,\mu)=b(x,\mu)$ and $\si(t,x,\mu)=\si(x,\mu)$ do not depend on $t$, and consider the exponential convergence  of the Markov process generated by $\tilde{\HL}=\tilde{\HL}_t$ on the Wasserstein space
$$ \scr P_2:=\bigg\{\mu\in \scr P:  \|\mu\|_2:=\bigg(\int_{\R^d} |x|^2\mu(\d x)\bigg)^{\ff 1 2 } <\infty\bigg\}.$$ To this end, we will take $\scr P_0=\tt{\scr P_0}=\scr P_2$,
which is a Polish space under the Wasserstein distance
 $$\W_2(\mu,\nu):= \inf_{\pi\in \C(\mu,\nu)} \bigg(\int_{\R^d\times\R^d} |x-y|^2\pi(\d x,\d y)\bigg)^{\ff 1 2},$$
 where $\C(\mu,\nu)$ is the set of couplings for $\mu$ and $\nu$.

 We will need the following linear growth and monotone conditions.
\beg{enumerate} \item[{\bf (A)}] $b,\bar b,\si,\bar \si$ are continuous on $\R^d\times\scr P_2$ and there exist constants $K,\ll,\kk,\bar\ll,\bar\kk\ge 0 $ such that for any $(x,\mu),(y,\nu)\in   \R^d\times \scr P_2,$ we have
\beq\label{A-1}  \big\{|b|+\|\si\| +|\bar b|+\|\bar \si\| \big\}(x,\mu) \le K (1+|x|+\|\mu\|_2),\end{equation}
 \beq\label{44LST}  2\<b(x,\mu)- b(y,\nu), x-y\> +\|\si(x,\mu)-\si(y,\nu)\|_{HS}^2
 \le \kk \W_2(\mu,\nu)^2 -\ll |x-y|^2,  \end{equation}
\beq\label{44LST'}  2\<\bar b(x,\mu)- \bar b(y,\nu), x-y\> +\|\bar\si(x,\mu)-\bar\si(y,\nu)\|_{HS}^2
 \le \bar\kk \W_2(\mu,\nu)^2 -\bar\ll |x-y|^2.  \end{equation}
  \end{enumerate}
According to \cite[Theorem 2.1]{RWRW18}, under {\bf (A)} both the SDE \eqref{eq:1.8} and the second equation in \eqref{44EE} are   well-posed for initial distributions in $\scr P_2$.
 We denote $P_t^*\zeta=\L_{X_t} $ for $X_t$ solving \eqref{eq:1.8} with $\L_{X_0}=\zeta\in \scr P_2$.
 Then by Theorem \ref{thm24}, the coupled equation \eqref{44EE} is well-posed for
 $$\scr P_0=\tt{\scr P_0}=\scr P_2,\ \ \scr A=\tt{\scr A}= C([0,\infty)\to\scr P_2),$$
 with $\mu_t= P_t^*\mu_0$ for $\mu_0=\zeta, t\ge 0.$ We denote by $(X_t^{\zeta,x})_{t\ge 0}$ the solution to the second equation in \eqref{44EE} starting at $x$. Then we have
 \beq\label{PPS} \d X_t^{\zeta,x}=  \bar b( X_t^{\zeta,x}, P_t^*\zeta)\d t+ \bar\si(  X_t^{\zeta,x},P_t^*\zeta)\d W_t,\ \ X_0^{\zeta,x}=x.\end{equation}

 Thus, Theorem \ref{prp41} implies condition \eqref{EQ432}, so that
   by Proposition \ref{P4.8} we see that \eqref{EQ433} gives a time-homogenous Markov transition kernel
\beq\label{44BMN}{\bf P}_t(x,\zeta;\cdot):={\bf P}_{0,t}(x,\zeta;\cdot)=\L_{X_t^{\zeta,x}}\times \dd_{P_t^*\zeta},\ \   t\ge  0, (x,\zeta)\in \R^d\times \scr P_2,\end{equation}
 generated by
 $\tilde{\HL}(=\tilde{\HL}_t,t\ge 0)$. For any $\nu,\mu\in \scr P_2$,   the distribution of the $\tilde{\HL}$-diffusion process $(X_t,\mu_t)$ at time $t$ with $\L_{X_0}=\theta$ and $\mu_0=\zeta$ is given by
 \beq\label{44BMO} \PP_t (\theta,\zeta;\cdot):=\int_{\R^d}  \PP_t (x,\zeta;\cdot)\theta(\d x)
 =  \L_{X_t^{\zeta,\theta}}\times \dd_{P_t^*\zeta},\ \ \zeta,\theta\in \scr P_2,t\ge 0,\end{equation} where
 $X_t^{\zeta,\theta}$ solves \eqref{44E2} for $s=0$, i.e.
\beq\label{*QP} \d X_t^{\zeta,\theta}=  \bar b(X_t^{\zeta.\theta}, P_t^*\zeta)\d t+ \bar\si( X_t^{\zeta,\theta},P_t^*\zeta)\d W_t,\ \ \L_{X_0^{\zeta,\theta}}=\theta. \end{equation}

 Let $\W_2^\rr$ be the $L^2$-Wasserstein distance induced by the following metric on $\R^d\times \scr P_2$:
 $$\rr((x,\mu), (y,\nu)):=\ss{|x-y|^2+\W_2(\mu,\nu)^2}.$$ Then  for any two probability measures $\LL_1,\LL_2$ on $E:=\R^d\times \scr P_2$,
 $$\W_2^\rr(\LL_1,\LL_2)^2:= \inf_{\Pi\in \C(\LL_1,\LL_2)} \int_{E\times E} \rr((x,\mu), (y,\nu))^2 \Pi((\d x,\d\mu); (\d y,\d\nu)),$$
 where $\C(\LL_1,\LL_2)$ is the set of all couplings of $\LL_1$ and $\LL_2$.

\beg{thm}\label{44T4.2} Assume {\bf (A)}. If   $\ll>\kk\ge 0$,
then $\PP_t$ has a unique invariant probability measure $\LL=\nu_\infty\times\dd_{\mu_\infty}$ for some $\mu_\infty,\nu_\infty\in \scr P_2$, such that  for any  $\zeta,\theta\in \scr P_2$,
$$ \W_2^{\rr} (  \PP_t (\theta,\zeta;\cdot), \LL)^2\le \W_2(\zeta,\mu_\infty)^2\bigg(\e^{-(\ll-\kk)t} +\ff{\bar\kk(\e^{-(\ll-\kk)t}-\e^{-\bar\ll t})}{\kk+\bar\ll-\ll}\bigg) +   \W_2(\theta,\nu_\infty)^2\e^{-\bar\ll t},\ \ t\ge 0,$$ where when $\kk+\bar\ll=\ll,$
    $$\ff{\e^{-(\ll-\kk)t}- \e^{-\bar\ll t} }{\kk+\bar\ll-\ll}:=  t\e^{-\bar\ll t},\ \ t\ge 0.$$
Consequently,   the  unique solution  $(\mu_t,\nu_t)_{t\ge 0}$ to $\eqref{eq:1.10}$
with $\mu_0,\nu_0\in \scr P_2$    satisfies
\beg{align*}&\W_2(\mu_t,\mu_\infty)^2+ \W_2(\nu_t,\nu_\infty)^2\\
&\le \W_2(\mu_0,\mu_\infty)^2\bigg(\e^{-(\ll-\kk)t} +\ff{\bar\kk(\e^{-(\ll-\kk)t}-\e^{-\bar\ll t})}{\kk+\bar\ll-\ll}\bigg) + \W_2(\nu_0,\nu_\infty)^2 \e^{-\bar\ll t},\ \ t\ge 0.\end{align*}   \end{thm}

\beg{proof} It suffices to prove the first assertion.

Firstly, by \cite[Theorem 3.1]{RWW18}, {\bf (A)}   implies that   $P_t^*$ has a unique invariant probability measure $\mu_\infty$ such that
\beq\label{44LST3} \W_2(P_t^*\zeta,\mu_\infty)^2\le \e^{-(\ll-\kk)t} \W_2(\zeta,\mu_\infty)^2,\ \ t\ge 0, \zeta\in \scr P_2.\end{equation}

Next, by \eqref{44LST'}, there exist constants $c_1,c_2>0$ such that
$$2\<\bar b(x,\mu_\infty), x\>+\|\bar\si(x,\mu_\infty)\|_{HS}^2\le c_1-c_2 |x|^2,\ \ x\in\R^d.$$ It is standard that this implies the existence of an invariant probability measure $ \nu_\infty$ of the diffusion process $\bar X_t$ associated with the SDE
\beq\label{NNE} \d \bar X_t= \bar b(\bar X_t,\mu_\infty)\d t+ \bar\si(\bar X_t,\mu_\infty) \d W_t.\end{equation}
Take an $\F_0$-measurable random variable $(\bar X_0, X_0^{\zeta,\theta})$ on $\R^d\times \R^d$, such that $\L_{(\bar X_0, X_0^{\zeta,\theta})}\in \scr C(\nu_\infty,\theta)$ and
\beq\label{*YYW} \E|\bar X_0-X_0^{\zeta,\theta}|^2 =\W_2(\theta,\nu_\infty)^2.\end{equation}
Let $(\bar X_t)_{t\ge 0}$ solve \eqref{NNE} with initial value $\bar X_0$ which has law $\nu_\infty$. Since $\nu_\infty$  is the invariant probability measure for the solution, we have
\beq\label{44IIN} \L_{\bar X_t}= \nu_\infty,\ \ t\ge 0.\end{equation}
So,
\beq\label{SPPN} \W_2(\L_{X_t^{\zeta,\theta}},\nu_\infty)^2\le \E|X_t^{\zeta,\theta}-\bar X_t|^2,\ \ t\ge 0.\end{equation}

On the other hand, let $\LL= \nu_\infty\times\dd_{\mu_\infty}$. Then \eqref{44BMO} implies
\beq\label{*WYY} \W_2({\bf P}_t(\theta,\zeta;\cdot),\LL)^2 \le \W_2(P_t^*\zeta,\mu_\infty)^2 +\W_2(\L_{X_t^{\zeta,\theta}}, \nu_\infty)^2,\ \ t\ge 0.\end{equation}
By \eqref{44LST'}, \eqref{*QP}, \eqref{NNE} and   It\^o's formula, we obtain
\beg{align*}  &\d |X_t^{\zeta,\theta}-\bar X_t|^2  \\
&= \Big\{2 \big\<\bar b(X_t^{\zeta,\theta}, P_t^*\zeta)- \bar b(\bar X_t, \mu_\infty), X_t^{\zeta,\theta}-\bar X_t\big\>+\|\bar \si(X_t^{\zeta,\theta}, P_t^*\zeta)- \bar \si(\bar X_t,\mu_\infty)\|_{HS}^2\Big\}\d t  +\d M_t\\
&\le \big\{\bar \kk \W_2(P_t^*\zeta,\mu_\infty)^2 -\bar \ll |X_t^{\zeta,\theta}-\bar X_t|^2\big\}\d t  +\d M_t,  \ \ t\ge 0,\end{align*}  where
$$\d M_t:= 2 \big\<(\bar\si(X_t^{\zeta,\theta}, P_t^*\zeta)-\bar\si(\bar X_t,\mu_\infty))\d W_t, X_t^{\zeta,\theta}-\bar X_t\big\>$$
is a martingale.  Combining this with \eqref{44LST3}, \eqref{44IIN} and \eqref{*YYW}, we get
\beq\label{44LST4}\beg{split} \E |X_t^{\zeta,\theta}-\bar X_t|^2&\le \e^{-\bar\ll t} \E|\bar X_0-X_0^{\zeta,\theta}|^2   +\bar \kk  \W_2(\zeta,\mu_\infty)^2
\int_0^t \e^{-\bar \ll(t-s)-(\ll-\kk)s} \d s \\
&=   \e^{-\bar\ll t} \W_2(\nu,\nu_\infty)^2   + \ff{\bar\kk   \W_2(\zeta,\mu_\infty)^2 }{\kk+\bar\ll-\ll} \big(\e^{-(\ll-\kk)t}- \e^{-\bar\ll t}\big),
  \end{split}\end{equation} where when $\kk+\bar\ll=\ll,$
    $$\ff{\e^{-(\ll-\kk)t}- \e^{-\bar\ll t} }{\kk+\bar\ll-\ll}:=\e^{-\bar\ll t}   \lim_{s\to \bar\ll} \ff{\e^{(\bar\ll-s)t}-1}{\bar\ll-s} = t\e^{-\bar\ll t}.$$
Combining \eqref{44LST4}  with \eqref{44LST3}, \eqref{SPPN} and \eqref{*WYY},   we derive
\beg{align*}  \W_2 (  \PP_t (\theta,\zeta;\cdot), \LL)^2\le
  \W_2(\zeta,\mu_\infty)^2\bigg(\e^{-(\ll-\kk)t} +\ff{\bar\kk(\e^{-(\ll-\kk)t}-\e^{-\bar\ll t})}{\kk+\bar\ll-\ll}\bigg) +\e^{-\bar\ll t} \W_2(\theta,\nu_\infty)^2,\ \ t\ge 0.\end{align*}
As a consequence, $\LL$ is the unique invariant probability measure for the $\HL$-diffusion process. \end{proof}

\section{Feynman-Kac formula for PDEs on $\R^d\times \scr P_2$}

In this section we aim to solve the PDE \eqref{44PDE} by using the diffusion process generated by $\tilde{\HL}_t$.
To this end, we   first recall the notion of intrinsic/$L$-derivative, and  present some classes of differentiable functions on $\scr P_2$, see Appendix below for a geometry explanation which implies that the class of cylindrical functions $\F C_b^2(\scr P)$ is included in $C_b^{(1,1)}(\scr P_2).$

\beg{defn}    Let  $f$ be a   real function on $\scr P_2$, and let $\Id:\R^d\to\R^d$ be the identity map. \beg{enumerate} \item[$(1)$] We call $f$ intrinsicly differentiable at $\mu\in \scr P_2$, if
$$L^2(\R^d\to\R^d;\mu)\ni \phi\mapsto \nn^{\scr P}_\phi f(\mu):=\lim_{\vv\downarrow 0} \ff{f(\mu\circ (\Id+\vv\phi)^{-1})- f(\mu)}\vv$$
  is a well defined bounded linear functional. In this case, the intrinsic derivative is the unique element
  $\nn^{\scr P}f(\mu)\in L^2(\R^d\to\R^d;\mu)$ such that
  $$\<\nn^{\scr P}f(\mu),\phi\>_{L^2(\mu)} = \nn^{\scr P}_\phi f(\mu),\ \ \phi\in  L^2(\R^d\to\R^d;\mu).$$     $f$ is called  intrinsicly differentiable if it is intrinsicly differentiable at any $\mu\in \scr P_2$.
  \item[$(2)$] We call $f$ $L$-differentiable on $\scr P_2$ if it is intrinsically differentiable and $$\lim_{\|\phi\|_{L^2(\mu)}\downarrow 0} \ff{|f(\mu\circ({\rm Id} +\phi)^{-1})-f(\mu)-\nn^{\scr P}_\phi f(\mu)|}{\|\phi\|_{L^2(\mu)}}=0,\ \ \mu\in \scr P_2.$$
 Let $C^1 (\scr P_2)$  be the set of all  $L$-differentiable functions $f: \scr P_2\to\R$ with $\nn^{\scr P}f(\mu)(y)$ having a   jointly continuous version in $(\mu,y)\in \scr P_2\times \R^d$, and denote $f\in C_b^1(\scr P_2)$ if moreover $ \nn^{\scr P}f(\mu)(y)$ is bounded.
     \item[$(3)$] We write $f\in C^{(1,1)} (\scr P_2)$ if $f\in C^{1}(\scr P_2)$ and $\nn^{\scr P} f(\mu)(y)$ is  differentiable in $y$, such that  $\nn \{\nn^{\scr P} f(\mu)(\cdot)\}(y)$ is jointly continuous in $(\mu,y)\in \scr P_2\times \R^d$. If moreover $f\in C_b^1(\scr P_2)$ and  $\nn \{\nn^{\scr P} f(\mu)(\cdot)\}(y)$ is bounded, we  denote $f\in  C_b^{(1,1)} (\scr P_2)$ \end{enumerate} \end{defn}

\subsection{Main result}

We   will work with the  following class $C_b^{0,2,(1,1)}([0,T]\times \R^d\times \scr P_2)$.

 \beg{defn} We write $f\in   C_b^{0,2,(1,1)}([0,T]\times\R^d\times\scr P_2)$, if $f(t,x,\mu)$ is   continuous in $(t,x,\mu)\in [0,T]\times\R^d\times \scr P_2$,  $C^2$ in $x\in \R^d$, $C^{(1,1)}$ in $\mu\in\scr P_2$, such that the derivatives
 $$\nn f(t,x,\mu),\ \ \nn^2f(t,x,\mu),\ \ \nn^{\scr P}f(t,x,\mu)(y),\ \ \nn\{\nn^{\scr P}f(t,x,\mu)(\cdot)\}(y)$$ are bounded and jointly continuous in $(t,x,\mu,y)\in [0,T]\times\R^d\times\scr P_2\times\R^d$.    If moreover $\pp_t f(t,x,\mu)$ is continuous in $(t,x,\mu)\in [0,T]\times \R^d\times \scr P_2$, we denote
 $f\in C_b^{1,2,(1,1)}([0,T]\times \R^d\times \scr P_2)$; and write $f\in C_b^{2,(1,1)}(\R^d\times \scr P_2)$  if $f(t,x,\mu)$ does not depend on $t$.
\end{defn}

According to Section 4, for any $(x,\mu)\in \R^d\times \scr P_2$, the $\tilde{\HL}_t$-diffusion  process  $ (X_{s,t}^{\mu,x},  \mu_t)_{t\ge s}$  starting at $(x,\mu)$  generated by $\tilde{\HL}_t$ on $\R^d\times\scr P_2$ can be constructed as follows: $\mu_t=P_{s,t}^*\mu$ is the law of $X_t$ which is the unique solution of \eqref{eq:1.8} from time $s$ with $\L_{X_s}=\mu$, and $(X_{s,t}^{\mu,x})_{t\ge s}$ is the unique solution to the second equation in \eqref{44EE} with $X_{s,s}^{\mu,x}=x$.

The main result in this section is the following.

\beg{thm}\label{44T3.1} Assume that $b, \si,\bar b,\bar\si\in    C_b^{0,2,2}([0,T]\times\R^d\times \scr P_2).$ Then for any
$\V,\f\in   C_b^{0,2,(1,1)}(\R^d\times\scr P_2)$ where $\V$ is bounded, and for any  $\Phi\in  C_b^{2,(1,1)}(\R^d\times\scr P_2)$, the PDE $\eqref{44PDE}$ with $u(T,\cdot,\cdot)= \Phi$ has a unique solution $u\in  C_b^{1,2,(1,1)}([0,T]\times \R^d\times \scr P_2)$, and the solution is given by
\beq\label{44FMM} \beg{split}&  u(t,x,\mu)\\
&=  \E \bigg[\Phi(X_{t,T}^{\mu,x},P_{t,T}^*\mu)\e^{\int_t^T \V(r,X_{t,r}^{\mu,x},P_{t,r}^*\mu )\d r}
 +\int_t^T \f(r,X_{t,r}^{\mu,x},P_{t,r}^*\mu)\e^{\int_t^r \V(\theta, X_{t,\theta}^{\mu,x},P_{t,\theta}^*\mu)\d \theta} \d r\bigg]\end{split}\end{equation} for $ (t,x,\mu)\in [0,T]\times \R^d\times\scr P_2.$
\end{thm}

When $b=\bar b$ and $\si=\bar\si$  do not depend on $t$ and ${\bf V}=0$, this result is included by  \cite[Theorem 9.2]{RWLJ} under slightly strongly conditions where the class $C_b^{0,2,(1,1)}$ is replaced by
 $C_{b,Lip}^{0,2,(1,1)}$: $f\in C_{b,Lip}^{0,2,(1,1)}$ means it is in $C_b^{0,2,(1,1)}$  such that  $$\nn f(t,x,\mu),\ \ \nn^2f(t,x,\mu),\ \ \nn^{\scr P}f(t,x,\mu)(y),\ \ \nn\{\nn^{\scr P}f(t,x,\mu)(\cdot)\}(y)$$ are Lipschtiz continuous in $(x,\mu)\in \R^d\times\scr P_2$. Moreover, \cite[Theorem 9.2]{RWLJ}  generalizes (with jump) and improves (under weaker conditions) the corresponding  earlier results in \cite{RWLP, RWChass, RWCM}.

\subsection{Proof of Theorem \ref{44T3.1}}

We first recall the following result taken from \cite{RWLJ} (see also \cite{RWLP, RWHL}), which was proved for time  independent coefficients $b=\bar b$ and $\si=\bar \si$, but the proof obviously works for the present time dependent  coefficients, since all calculations therein only rely on the regularity of coefficients in the space-distribution variables $(x,\mu)$ but  has nothing to do with derivatives in time.

\beg{lem}\label{44L00} Let  $b,\si,\bar b,\bar \si \in C_b^{0,2,2}([0,T]\times\R^d\times \scr P_2)$. Then  $$ \nn X_{s,t}^{\mu,\cdot}(x),\   \nn^2 X_{s,t}^{\mu,\cdot}(x),\ \nn^{\scr P}X_{s,t}^{\cdot,x}(\mu)(y),\ \nn\{\nn^{\scr P} X_{s,t}^{\cdot,x}(\mu)(\cdot)\}(y)$$ are  jointly  continuous  in
$(t,x,\mu,y)\in [s,T]\times\R^d\times\scr P_2\times\R^d,$  and  there exists a constant $c>0$ such that
$$  \E \big(\|\nn X_{s,t}^{\mu,\cdot}(x)\|^2+ \|\nn^2 X_{s,t}^{\mu,\cdot}(x)\|^2+ \|\nn^{\scr P} X_{s,t}^{\cdot,x}(\mu)\|_{L^2(\mu)},\ \|\nn\{\nn^{\scr P} X_{s,t}^{\cdot,x}(\mu)\}\|_{L^2(\mu)} \big)\le c,$$
  holds for all $0\le s\le T$ and $(x,\mu)\in \R^d\times \scr P_2$.\end{lem}

Replacing $(\bar b,\bar \si)$ in \eqref{44E2} by $(b,\si)$, we consider the SDE
 \beq\label{44E*} \d Y_{s,t}^{\mu,x}= b(t, Y_{s,t}^{\mu,x}, P_{s,t}^*\mu)\d t + \si(t, Y_{s,t}^{\mu,x}, P_{s,t}^*\mu)\d W_t,\ \ Y_{s,s}^{\mu,x}=x.\end{equation}  Then   the Markov property of the solution implies
  \beq\label{44MKK} P_{s,t}^*\mu=\int_{\R^d} \L_{Y_{s,t}^{\mu,x}}\mu(\d x),\end{equation}where
 $P_{s,t}^*\mu:= \L_{X_t}$ for $X_t$ solving \eqref{eq:1.8} from time $s$     with $\L_{X_s}=\mu$.
 Combining this with Lemma \ref{44L00}, which applies to $Y_{s,t}^{\mu,x}$ replacing $X_{s,t}^{\mu,x}$ by taking $(\bar b,\bar\si)=(b,\si)$,  we have the following result.

 \beg{lem}\label{44LM} Let $b,\si  \in    C_b^{0,2,(1,1)}([0,T]\times\R^d\times \scr P_2)$. Then the assertion in Lemma $\ref{44L00} $ holds for $Y_{s,t}^{\mu,x}$ replacing $X_{s,t}^{\mu,x},$  and
 for any $f\in C_b^{(1,0)}(\scr P_2)$,
\beq\label{44DRR2} \beg{split} \nn^{\scr P} f(P_{s,t}^* \cdot)(\mu)(y)= & \int_{\R^d} \E\Big[ \big\{\nn^{\scr P}Y_{s,t}^{\cdot,z}(\mu)(y)\big\}(\nn^{\scr P} f)(P_{s,t}^*\mu)(Y_{s,t}^{\mu,z})\Big]\mu(\d z)\\
&+ \E \Big[\big\{\nn Y_{s,t}^{\mu,\cdot}(y)\big\}(\nn^{\scr P} f)(P_{s,t}^*\mu)(Y_{s,t}^{\mu,y})\Big],\end{split}\end{equation}
where for  vectors $v_1,v_2\in\R^d$,
 $$\big\<\{\nn Y_{s,t}^{\mu,\cdot}(z)\} v_1, v_2\big\>:= \big\<\nn_{v_2} Y_{s,t}^{\mu,\cdot}(z), v_1\big\>,$$
 and $\big\{\nn^{\scr P}Y_{s,t}^{\cdot,z}(\mu) (\cdot) \big\} v_1\in L^2(\R^d\to \R^d;\mu)$ is defined by
 $$ \big\<  \big\{\nn^{\scr P}Y_{s,t}^{\cdot,z}(\mu) (\cdot) \big\} v_1,\phi\big\>_{L^2(\mu)}:= \big\<\nn^{\scr P}_\phi \big\{ Y_{s,t}^{\cdot,z}(\mu) (\cdot) \big\},v_1\big\>,\ \ \phi\in L^2(\R^d\to\R^d;\mu).$$
 \end{lem}
\beg{proof} Let $f\in C_b^{(1,0)}(\scr P_2),$ i.e. $f$ is $L$-differentiable with
$\nn^{\scr P}f(\mu)(y)$ having a bounded jointly continuous version. For a family of random variables $\{\xi^\vv: \vv \in [0,1)\}$ with $\dot \xi^0:=\ff{\d\xi^\vv}{\d\vv}\big|_{\vv=0}$ existing in $L^1(\P)$, we have
\beq\label{44DRR} \lim_{\vv\downarrow 0} \ff{ f(\L_{\xi(\vv)})- f(\L_{X^0})} \vv   = \E\<\nn^{\scr P} f(\L_{\xi(0)})(\xi(0)), \dot\xi^0\>.\end{equation}
See for instance   \cite[Proposition 3.1]{RWRW18b}, which is slightly extended from \cite[Proposition A.2]{RWHSS}.
Since   $P_{s,t}^*\mu= \int_{\R^d} \L_{Y_{s,t}^{\mu,z}}\mu(\d z),$ it follows form \eqref{44DRR} that for any $\phi\in L^2(\mu)$,
\beg{align*}\nn^{\scr P}_\phi f (P_{s,t}^*\cdot)(\mu) &=\ff{\d}{\d\vv} \bigg\{ f\bigg(\int_{\R^d} \L_{Y_{s,t}^{\mu\circ({\rm Id}+\vv \phi)^{-1},z}}\mu(\d z)\bigg)
  + \ff{\d}{\d\vv}  f\bigg(\int_{\R^d} \L_{Y_{s,t}^{\mu, z+\vv\phi(z)}}\mu(\d z)\bigg)\bigg\}\bigg|_{\vv=0} \\
&= \int_{\R^d}  \E \big\<(\nn^{\scr P}f)(P_{s,t}^*\mu)(Y_{s,t}^{\mu,z}), \nn^{\scr P}_\phi Y_{s,t}^{\cdot,z}(\mu)+ \nn_{\phi(z)} Y_{s,t}^{\mu,\cdot}(z) \big\> \mu(\d z) \\
&=  \E\int_{\R^d\times\R^d} \Big[\big\<\big\{\nn^{\scr P}  Y_{s,t}^{\cdot,z}(\mu) (\cdot) \big\} (\nn^{\scr P}f)(P_{s,t}^*\mu)(Y_{s,t}^{\mu,z}),\phi\>_{L^2(\mu)}\\
 &\qquad + \big\<\{\nn Y_{s,t}^{\mu,\cdot}(z)\} (\nn^{\scr P}f)(P_{s,t}^*\mu)(Y_{s,t}^{\mu,z}), \phi(z)\big\>\Big]\mu(\d z).
 \end{align*}
 Therefore, $f(P_{s,t}^*\cdot)$ is  intrinsic differentiable such that \eqref{44DRR2} holds. \end{proof}

 \beg{lem}\label{44LM3}  Assume that $b,\si,\bar b, \bar\si\in   C_b^{0,2,(1,1)}([0,T]\times\R^d\times \scr P_2)$. For $\V,\f\in C_b^{0,2,(1,1)}([0,T]\times\R^d\times \scr P_2)$,    let $u$ be in $\eqref{44FMM}$. Then $u\in  C_b^{0,2,(1,1)}([0,T]\times \R^d\times \scr P_2)$.\end{lem}

\beg{proof}  The proof is   more or less standard, but for  completeness we present below a brief proof.

Obviously, $u(t,x,\mu)$ is joint continuous in $(t,x,\mu)$. Next,
  by making derivatives in $(x,\mu)$ for $u$ in \eqref{44FMM}, we obtain
\beg{align*}  \nn^i u(t,\cdot,\mu)(x)= &\E\bigg[\e^{\int_t^T \V(X_{t,r}^{\mu,x}, P_{t,r}^*\mu)\d r}\bigg(\<\nn^i \Phi(\cdot,P_{s,T}^*\mu)(X_{t,T}^{\mu,x}),\nn^i X_{t,T}^{\mu,\cdot}(x)\>\\
&  + \Phi(X_{t,T}^{\mu,x},P_{s,T}^*\mu)\int_t^T \<\nn^i \V(\cdot,P_{s,r}^*\mu)(X_{t,r}^{\mu,x}),\nn^i X_{t,r}^{\mu,\cdot}(x)\>\d r\bigg) \\
&  + \int_t^T  \e^{\int_t^r \V(\theta, X_{t,\theta}^{\mu,x}, P_{t,\theta}^*\mu)\d \theta}\bigg((\<\nn^i \f(\cdot,P_{s,r}^*\mu)(X_{t,r}^{\mu,x}),\nn^i X_{t,r}^{\mu,\cdot}(x)\> \d r \\
& +   \f(X_{t,r}^{\mu,x},P_{s,r}^*\mu)  \int_t^r \big\<\nn^i \V(\theta, \cdot,  P_{t,\theta}^*\mu)(X_{t,\theta}^{\mu,x})\d \theta,\nn^i X_{t,\theta}^{\mu,\cdot}(x)\big\>\d\theta \bigg)\d r \bigg],\ \ i=1,2.\end{align*}
and
\beg{align*} & \nn^{\scr P} u(t,x,\cdot)(\mu)(y)\\
&= \E\bigg[\e^{\int_t^T \V(X_{t,r}^{\mu,x}, P_{t,r}^*\mu)\d r}\bigg(\big\<\nn^{\scr P} \Phi(X_{t,T}^{\cdot,x}, P_{t,T}^*\mu)(\mu)(y),\nn^{\scr P} X_{t,T}^{\cdot,x}(\mu)(y)\big\>  \\
&  + \Phi(X_{t,T}^{\mu,x},P_{s,T}^*\mu)\int_t^T \big\<\nn^{\scr P} \V(X_{t,r}^{\cdot,x},P_{s,r}^*\mu)(\mu)(y),\nn^{\scr P} X_{t,r}^{\cdot,x}(\mu) (y)\big\>\d r \\
& + \nn^{\scr P}\Phi(X_{s,t}^{x,\mu}, P_{s,t}^*\cdot)(\mu)(y) + \Phi(X_{t,T}^{x,\mu},P_{s,T}^*\mu)\int_t^T \nn^{\scr P} \V(X_{t,r}^{x,\mu},P_{s,r}^*\cdot)(\mu)(y)  \d r \bigg)\\
 &  +    \int_t^T \e^{\int_t^r \V(\theta, X_{t,\theta}^{\mu,x}, P_{t,\theta}\mu)\d\theta} \bigg(\big\<\nn^{\scr P} \f(X_{t,r}^{\cdot,x},P_{s,r}^*\mu)(\mu),\nn^{\scr P} X_{t,r}^{\cdot,x}(\mu)\big\>(y)+\nn^{\scr P} \f (r,X_{s,r}^{\mu,x},P_{s,r}^*\cdot)(\mu)(y)\\
 &+ \f(r,X_{t,r}^{\mu,x},P_{t,r}^*\mu)\int_t^r \big\{\nn^{\scr P} \V(\theta, X_{t,\theta}^{\cdot,x}, P_{t,\theta}\mu )(\mu)(y) + \nn^{\scr P} \V(\theta, X_{t,\theta}^{\mu,x}, P_{s,\theta}\cdot)(\mu)(y)\big\} \d\theta\bigg) \d r  \bigg].
 \end{align*} Combining these with Lemma \ref{44LM}, we  finish the proof.
\end{proof}

We will need the following lemma, which follows from   \cite[Theorem 3.3]{RWChass} or \cite[Proposition A.6]{RWHSS} under the  stronger  condition
\beq\label{44*AD}  \int_s^T   \E\big(|\aa_t|^2+\|\bb_t\|^4\big)\d t <\infty,\ \ T\in (s,\infty),\end{equation}
where $\|\cdot\|$ is the operator norm of matrices.

\beg{lem}\label{44LLN} Let $\aa: [s,\infty)\to \R^d$ and $\bb: [s,\infty)\to\R^{d\otimes m}$ be progressively measurable with
\beq\label{44**AD} \E\bigg(\int_s^T  |\aa_t|\d t\bigg)^2+\E\int_s^T\|\bb_t\|^2 \d t <\infty,\ \ T\in (s,\infty).\end{equation} For $X_s\in L^2(\OO\to\F_s;\P)$, let
$\mu_t=\L_{X_t}$ for  $$X_t:=X_s+\int_s^t \aa_r\d r+ \int_s^t \bb_r\d W_r,\ \ t\ge s.$$ Then $\mu_\cdot\in C([s,\infty)\to\scr P_2)$ and    for any $f\in   C_b^{(1,1)}(\scr P_2)$,
$$\ff{ \d f(\mu_t)}{\d t}= \E   \Big[\ff 1 2  \big\<\bb_t\bb_t^*, \nn\{\nn^{\scr P} f (\mu_t)(\cdot)\}(X_t)\big\> _{HS}
 +\big\<\aa_t,  (\nn^{\scr P} f(\mu_t))(X_t)\big\>\Big],\ \ t\ge s.$$
\end{lem}
\beg{proof} Since $\mu_s\in \scr P_2$ and \eqref{44**AD} holds, it is easy to see that $\mu_\cdot\in C([s,\infty)\to\scr P_2)$.  For any $n\ge 1$, let $\aa_t^n=\aa_t1_{\{|\aa_t|\le n\}},\ \bb_t^n=\bb_t1_{\{\|\bb_t\|\le n\}},$ and let $\mu_t^n=\L_{X_t^n}$ for
$$X_t^n:=X_s+\int_s^t \aa_r^n\d r+ \int_s^t \bb_r^n\d W_r,\ \ t\ge s.$$
 Then
$$\lim_{n\to\infty}\sup_{t\in [s,T]} W_2(\mu_t^n,\mu_t)=0$$ and   \eqref{44*AD} holds for $(\aa_t^n,\bb_t^n)$ replacing $(\aa_t,\bb_t)$. So, by
 \cite[Theorem 3.3]{RWChass} or \cite[Proposition A.6]{RWHSS}, we obtain
 $$f(\mu_t^n)= f(\mu_s) +\E\int_s^t   \Big[\ff 1 2  \big\<\bb_r^n\{\bb_r^n\}^*, \nn\{\nn^{\scr P}f (\mu_r^n)(\cdot)\}(X_r^n)\big\> _{HS}
 +\big\<\aa_r^n,  (\nn^{\scr P} f(\mu_r^n)(X_r^n))\big\>\Big]\d r$$   for $t\ge s.$
Since $\nn^{\scr P}f(\mu)(y)$ and $\nn\{\nn^{\scr P}f(\mu)(\cdot)\}(y)$ are bounded and continuous in $(\mu,y)$, by \eqref{44**AD}  we may apply the dominated convergence theorem with $n\to\infty$ to derive   $$f(\mu_t)= f(\mu_s) +\E\int_s^t    \Big[\ff 1 2  \big\<\bb_r\bb_r^*, \nn\{\nn^{\scr P}f (\mu_r)(\cdot)\}(X_r)\big\> _{HS}
 +\big\<\aa_r,  (\nn^{\scr P} f(\mu_r)(X_r))\big\>\Big]\d r,\ \ t\ge s.$$ Then the proof is finished.
\end{proof}
We now  apply Lemma \ref{44LLN} to prove the following It\^o's formula  for $(X_{s,t}^{x,\mu}, P_{s,t}^*\mu)$.

 \beg{lem}\label{44LM2} Assume that there exists an increasing function $K: [0,\infty)\to (0,\infty)$ such that
\beq\label{AAA}  \big\{|b|+\|\si\| +|\bar b|+\|\bar \si\| \big\}(t,x,\mu) \le K(t) (1+|x|+\|\mu\|_2),\ \ (t,x,\mu)\in [0,\infty)\times \R^d\times \scr P_2.\end{equation}
    Then for any $f\in  C^{2,(1,1)}(\R^d\times \scr P_2)$ and   $s\in [0,T)$,
 $$\d f (X_{s,t}^{\mu,x}, P_{s,t}^*\mu) = \tilde{\HL}_T f(X_{s,t}^{\mu,x}, P_{s,t}^*\mu)\d t + \<\nn f(\cdot, P_{s,t}^*\mu)(X_{s,t}^{\mu,x}), \si(t, X_{s,t}^{\mu,x},P_{s,t}^*\mu)\d W_t\>,\ \ t\in [s,T].$$
 \end{lem}
 \beg{proof} It is easy to see that  \eqref{AAA}  implies \eqref{44**AD} for
 $$\aa_t:= b(t,X_t, P_{s,t}^*\mu),\ \ \bb_t:=\si(t,X_t,P_{s,t}^*\mu)$$ for $X_t$ solving \eqref{eq:1.8} from time $s$ with $\L_{X_s}=\mu$. By Lemma \ref{44LLN}
 and the definition of $\HL_t$  in \eqref{eq:1.7},  for any $z\in\R^d$ we have
 $$\d f(z, P_{s,t}^*\mu)= \HL_t f(z, P_{s,t}^*\mu)\d t,\ \ t\ge s.$$ Combining this with  It\^o's formula for $X_{s,t}^{x,\mu}$ in \eqref{44E2}, we obtain
 \beg{align*} \d f(X_{s,t}^{\mu,x}, P_{s,t}^*\mu)& = \big\{\d f(z, P_{s,t}^*\mu)\big\}\big|_{z=X_{s,t}^{\mu,x}} + \big\{\d f(X_{s,t}^{\mu,x},\nu)\big\}\big|_{\nu= P_{s,t}^*\mu}\\
 &=  \HL_t f(X_{s,t}^{\mu,x}, P_{s,t}^*\mu)\d t + \<\nn f(\cdot, P_{s,t}^*\mu)(X_{s,t}^{\mu,x}), \si(t, X_{s,t}^{\mu,x},P_{s,t}^*\mu)\d W_t\>.\end{align*}  \end{proof}

We are now ready to prove Theorem \ref{44T3.1} as follows.

\beg{proof}[Proof of Theorem \ref{44T3.1}] {\bf (a)} We first prove that $u$ in \eqref{44FMM} solves \eqref{44PDE}. Let $(t,x,\mu)\in [0,T]\times \R^d\times\scr P_2$. For any $\vv\in (0,T-t)$ we have
\beg{align*} u(t,x,\mu)&:= \E\bigg[\Phi(X_{t,T}^{x,\mu}, P_{t,T}^*\mu)\e^{\int_t^T\V(r,X_{t,r}^{\mu,x}, P_{t,r}^*\mu)\d r} +\int_t^T \e^{\int_t^r \V(\theta, X_{t,\theta}^{\mu,x}, P_{t,\theta}^*\mu)\d \theta}\f(r, X_{t,r}^{x,\mu}, P_{t,r}^*\mu)\d r\bigg]\\
&= I_1+I_2,\end{align*} where
\beg{align*} &I_1:=   \E\bigg[\Phi(X_{t,T}^{\mu,x}, P_{t,T}^*\mu)\e^{\int_{t+\vv}^T\V(r,X_{t,r}^{\mu,x}, P_{t,r}^*\mu)\d r} +\int_{t+\vv}^T \e^{\int_{t+\vv}^r \V(\theta, X_{t,\theta}^{\mu,x}, P_{t,\theta}^*\mu)\d \theta} \f(r, X_{t,r}^{\mu,x}, P_{t,r}^*\mu)\d r\bigg],\\
&I_2:= \E\bigg[\Phi(X_{t,T}^{\mu,x}, P_{t,T}^*\mu)\Big\{\e^{\int_t^T\V(r,X_{t,r}^{\mu,x}, P_{t,r}^*\mu)\d r}- \e^{\int_{t+\vv}^T\V(r,X_{t,r}^{\mu,x}, P_{t,r}^*\mu)\d r} \Big\} \\
&\qquad \qquad +\int_t^{t+\vv}\e^{\int_t^r \V(\theta, X_{t,\theta}^{\mu,x}, P_{t,\theta}^*\mu)\d \theta}  \f(r, X_{t,r}^{\mu,x}, P_{t,r}^*\mu)\d r\\
&\qquad \qquad+ \int_{t+\vv}^T \Big\{\e^{\int_{t}^r \V(\theta, X_{t,\theta}^{\mu,x}, P_{t,\theta}^*\mu)\d \theta} - \e^{\int_{t+\vv}^r \V(\theta, X_{t,\theta}^{\mu,x}, P_{t,\theta}^*\mu)\d \theta}\Big\} \f(r, X_{t,r}^{\mu,x}, P_{t,r}^*\mu)\d r\bigg].\end{align*}
By the Markov property of $(X_{t,r}^{\mu,x}, P_{t,r}^*\mu)_{r\in [t,T]}$, we obtain
\beg{align*} I_1&= \E\bigg\{\E\bigg[\Phi(X_{t+\vv,T}^{\nu,y}, P_{t+\vv,T}^*\nu)\e^{\int_{t+\vv}^T \V(r,X_{t+\vv,r}^{\nu,y}, P_{t+\vv,r}^*\nu)\d r}\\
&\qquad \qquad + \int_{t+\vv}^T \f(r,X_{t+\vv,r}^{\nu,y}, P_{t+\vv,r}^*\nu)\e^{\int_{t+\vv}^r \V(\theta,X_{t+\vv,\theta}^{\nu,y}, P_{t+\vv,\theta}^*\nu)\d\theta}\d r\bigg]\bigg|_{(y,\nu)=(X_{t,t+\vv}^{\mu,x}, P_{t,t+\vv}^*\mu)}\bigg\}\\
&= \E u(t+\vv, X_{t,t+\vv}^{\mu,x}, P_{t,t+\vv}^*\mu).\end{align*} Combining this with Lemma \ref{44LM3} and Lemma \ref{44LM2}, we arrive at
\beg{align*} I_1  = u(t+\vv, x,\mu)+\E \int_t^{t+\vv} \tilde{\HL}_r u(t+\vv, \cdot,\cdot)(X_{t,r}^{\mu,x}, P_{t,r}^*\mu) \d r.\end{align*}
Noting that $u(t,x,\mu)= I_1+I_2$,     $b,\si,\Phi$ and $ \f$ are continuous with linear growth,  $\V$ is continuous and bounded, and $u\in C_b^{0,2,(1,1)}([0,T]\times \R^d\times \scr P_2)$, we may apply the dominated convergence theorem to derive
\beg{align*} &\lim_{\vv\downarrow 0}\ff{u(t,x,\mu)- u(t+\vv,x,\mu)}\vv \\
&=  \lim_{\vv\downarrow 0}\ff 1 \vv \E\bigg\{\int_t^{t+\vv} \tilde{\HL}_r u(t+\vv, \cdot,\cdot)(X_{t,r}^{\mu,x}, P_{t,r}^*\mu) \d r + I_2\bigg\}\\
&= \tilde{\HL}_t u(t,\cdot,\cdot)(x,\mu) + (u\V)(t,x,\mu)+ \f(t,x,\mu).\end{align*}
Therefore, $u$ solves \eqref{44PDE}, and $\pp_t u$ is continuous on $[0,T]\times \R^d\times \scr P_2)$, so that by Lemma \ref{44LM3} and the definition, we have $u\in C_b^{1,2,(1,1)}([0,T]\times \R^d\times\scr P_2)$.

{\bf (b)} Let $u\in    C^{1,2,(1,1)}([0,T]\times \R^d\times\scr P_2)$  be a solution to \eqref{44PDE}, we prove that it satisfies \eqref{44FMM}. Indeed, let
$$\eta_t= u(t, X_{s,t}^{\mu,x}, P_{s,t}^*\mu)\e^{\int_s^t \V(r,X_{s,r}^{\mu,x}, P_{s,r}^*\mu)\d r} + \int_s^t \f(r,X_{s,r}^{\mu,x}, P_{s,r}^*\mu)
\e^{\int_s^r \V(\theta,X_{s,\theta}^{\mu,x}, P_{s,\theta}^*\mu)\d \theta}\d r,\ \ t\in [0,T].$$ By Lemma \ref{44LM2} and \eqref{44PDE}, for any $s\in [0,T)$,  we have
\beg{align*} \d \eta_t
&=\Big\{(\pp_t+\tilde{\HL}_t) u(t, \cdot,\cdot)(X_{s,t}^{\mu,x}, P_{s,t}^*\mu) + (u\V)(t,X_{s,t}^{\mu,x}, P_{s,t}^*\mu)+\f(t,X_{s,t}^{\mu,x}, P_{s,t}^*\mu)\\
&\qquad \qquad \qquad  + \big\< \nn u(t, \cdot, P_{s,t}^*\mu)(X_{s,t}^{\mu,x}), \si(t, X_{s,t}^{\mu,x}, P_{s,t}^*\mu) \d W_t\big\>\Big\}\e^{\int_s^t \V(r,X_{s,r}^{\mu,x}, P_{s,r}^*\mu)\d r}\\
&= \e^{\int_s^t \V(r,X_{s,r}^{\mu,x}, P_{s,r}^*\mu)\d r}\big\< \nn u(t, \cdot, P_{s,t}^*\mu)(X_{s,t}^{\mu,x}), \si(t, X_{s,t}^{\mu,x}, P_{s,t}^*\mu) \d W_t\big\>,\ \ \ t\in [s,T].\end{align*}
 Therefore, for any $s\in [0,T]$,
\beg{align*} &u(s,x,\mu)= \E \eta_s=\E\eta_T\\
&= \E\bigg\{u(T, X_{s,T}^{\mu,x}, P_{s,T}^*\mu)\e^{\int_s^T \V(r,X_{s,r}^{\mu,x}, P_{s,r}^*\mu)\d r} + \int_s^T \f(r,X_{s,r}^{\mu,x}, P_{s,r}^*\mu)
\e^{\int_s^r \V(\theta,X_{s,\theta}^{\mu,x}, P_{s,\theta}^*\mu)\d \theta}\d r\bigg\},\end{align*}
that is, $u$ satisfies \eqref{44FMM}.
\end{proof}
\appendix
\section{ A natural tangent bundle over $\scr P $}
It is well-known how to identify natural tangent bundles over $``$manifold-like" state spaces $M$ and their corresponding gradients. For this one has to fix a large enough space of $``$smooth" real-valued functions $\scr F$ ($``$test functions") on $M$ and for each $x\in M$ a set of $``$suitable curves" $\gamma^x:[0,1]\rightarrow M,\ \gamma(0)=x$ along which we can differentiate $t\mapsto f(\gamma^x(t))$ at $t=0$ for all $f\in\scr F$. This construction has been performed in \cite{RWAKR}, \cite{AKR2} (see also \cite{RWAKR2}, \cite{RWAKR157}, \cite{MR}, \cite{RWORS}, \cite{RWOR}, \cite{RSAM}, \cite{RWR98}) with the space $M$ being the space $\Gamma$ of all $\mathbb{Z}_+$-valued Radon-measures on a Riemannian manifold, i.e. $\Gamma$ is the \textit{configuration space} over this manifold. The space $\scr F$ there consists of so-called finitely based smooth functions on $\Gamma$. It turns out that in this case the resulting tangent bundle $(T_\mu\Gamma)_{\mu\in\Gamma}$ consists of linear spaces $T_\mu\Gamma$ given by $\mu$-square integrable sections over the underlying Riemannian manifold. Let us now present the completely analogous construction in the case we are concerned with in this paper, where the underlying Riemannian manifold is $\R^d$ and $\Gamma$ is replaced by $\scr P(\R^d) =:\scr P$. Define (see \eqref{44FC})
\begin{align*}
\scr F:=\scr F C_b^2(\scr P)
\end{align*}
and for $\mu\in\scr P $, $\phi\in L^2(\R^d\rightarrow\R^d,\mu):$
\begin{align}\label{A.1}
\gamma_\phi^\mu(t):=\mu\circ (\operatorname{Id}+t\phi)^{-1},\quad t\geq0.
\end{align}
So, in our case the set of $``$suitable curves" starting at $\mu\in\scr P $ are labelled by $\phi\in L^2(\R^d\rightarrow\R^d,\mu)$.\\\\
\underline{Claim:} The resulting tangent bundle over $\scr P $ is $(T_\mu\scr P)_{\mu\in\scr P}:=(L^2(\R^d\rightarrow\R^d,\mu))_{\mu\in\scr P}$\begin{proof}
Let $F\in\scr F=\scr FC^2_b(\scr P)$ and $\mu\in\scr P$. Then
\begin{align}\label{eq:A.1}
F(\mu)=f(\mu(h_1),\dots,\mu(h_n)),
\end{align}
for some $n\in\N$, $h_1,\dots h_n\in C_b^2(\R^d)$, $f\in C^1_b(\R^n)$ and thus by the chain rule for all $\phi\in L^2(\R^d\rightarrow\R^d,\mu)$
\begin{align}\label{eq:A.2}
\frac{\d}{\d t}F&(\gamma_\phi^\mu(t))_{|t=0}\notag\\
&=\sum_{i=1}^n(\partial_if)(\mu(h_1),\dots,\mu(h_n))\mu(\langle\nabla h_i,\phi\rangle_{\R^d})\notag\\
&=\Big\langle\sum_{i=1}^n(\partial_if)(\mu(h_1),\dots,\mu(h_n))\nabla h_i,\phi\rangle_{L^2(\R^d\rightarrow\R^d,\mu)}
\end{align}
\end{proof}
Define the corresponding gradient for $F\in\scr FC_b^2(\scr P)$ by
\begin{align}\label{eq:A.22}
\nabla^\scr P F(\mu):=\sum_{i=1}^n(\partial_if)(\mu(h_1),\cdots,\mu(h_n))\nabla h_i\in L^2(\R^d\rightarrow\R^d,\mu).
\end{align}
Then by \eqref{eq:A.2} we have for all $\phi\in L^2(\R^d\rightarrow\R^d,\mu)$
\begin{align}\label{eq:A.3}
\frac{\d}{\d t}F(\gamma_\phi^\mu(t))_{|{t=0}}=\langle\nabla^\scr PF(\mu),\phi\rangle_{L^2(\R^d\rightarrow\R^d,\mu)}.
\end{align}
In particular, the definition of $\nabla^\scr PF$ is independent of the particular representation of F in \eqref{eq:A.2}.
\begin{rem}
We note that for $F\in\scr FC_b^2(\scr P)$ we have that $\nabla^\scr PF$ is bounded on $\R^d$. So, the right hand side  of \eqref{eq:A.3}  is well-defined also if merely $\phi\in L^1(\R^d\rightarrow\R^d,\mu)$. For simplicity of notation and extending the inner product we keep the notation
\begin{align*}
\int\langle\nabla^\scr PF(\mu),\phi\rangle_{\R^d}d\mu=:\langle\nabla^\scr PF(\mu),\phi\Big\rangle_{L^2(\R^d\rightarrow\R^d,\mu)}
\end{align*}
also in this case in the entire paper.
\end{rem}

\section*{Acknowledgements} We would like to thank Marco Rehmeier for pointing out serveral misprints and small errors in an earlier version of this paper.

\end{document}